\documentclass[12pt,a4paper,reqno]{amsart}

\usepackage{amsfonts}
\usepackage{amsmath}
\usepackage{amssymb}
\usepackage{amsthm}
\usepackage{amsxtra}
\usepackage{bbm}
\usepackage{braket}
\usepackage{comment}
\usepackage{extarrows}
\usepackage{graphics}
\usepackage{latexsym}
\usepackage{mathabx}
\usepackage{mathrsfs}
\usepackage{mathtools}
\usepackage{verbatim}
\usepackage[normalem]{ulem}

\usepackage{bm}
\allowdisplaybreaks[4]

\usepackage{tikz}
\usetikzlibrary{patterns}
\usetikzlibrary{decorations.pathreplacing, decorations.markings}
\tikzset{->-/.style={decoration={markings,mark=at position #1 with {\arrow{>}}},postaction={decorate}}}

\usetikzlibrary{arrows}
\usetikzlibrary{patterns,patterns.meta}

\usepackage[initials,lite]{amsrefs} 
\AtBeginDocument{%
    \def\MR#1{}
}

\usepackage{cleveref}

\usepackage{a4wide}
\theoremstyle{plain}
\newtheorem{Theorem}{Theorem}[section]
\newtheorem{Lemma}[Theorem]{Lemma}
\newtheorem{Corollary}[Theorem]{Corollary}
\newtheorem{Proposition}[Theorem]{Proposition}

\newtheorem{Observation}[Theorem]{Observation}

\theoremstyle{definition}

\newtheorem{Assumptions and Discussion}[Theorem]{Assumptions and Discussion}

\newtheorem{Example}[Theorem]{Example}
\newtheorem{Definition}[Theorem]{Definition}

\newtheorem{Remark}[Theorem]{Remark}

\theoremstyle{remark}

\newtheorem{Setting}[Theorem]{Setting}

\newtheorem*{acknowledgment*}{Acknowledgment}

\usepackage[shortlabels]{enumitem}
\SetEnumerateShortLabel{a}{\textup{(\alph*)}}
\SetEnumerateShortLabel{A}{\textup{(\Alph*)}}
\SetEnumerateShortLabel{1}{\textup{(\arabic*)}}
\SetEnumerateShortLabel{i}{\textup{(\roman*)}}
\SetEnumerateShortLabel{I}{\textup{(\Roman*)}}

\definecolor{amaranth}{rgb}{0.9, 0.17, 0.31}
\definecolor{brown(web)}{rgb}{0.65, 0.16, 0.16}
\definecolor{darkcyan}{rgb}{0.0, 0.55, 0.55}
\definecolor{dodgerblue}{rgb}{0.12, 0.56, 1.0}

\def\codim{\operatorname{codim}}

\def\deg{\operatorname{deg}}
\def\depth{\operatorname{depth}}

\def\dim{\operatorname{dim}}

\def\e{\operatorname{e}}

\def\gens{\operatorname{gens}}

\def\Hilb{\operatorname{Hilb}}

\def\ini{\operatorname{in}} 

\def\ker{\operatorname{ker}}
\def\KK{{\mathbb K}}

\def\link{\operatorname{link}}

\def\NN{{\mathbb N}}

\def\uotimes{\mathbin{\underline{\otimes}}}

\def\ZZ{{\mathbb Z}}

\newcommand\bdu{{\bm u}}

\newcommand\bdv{{\bm v}}

\newcommand\bdw{{\bm w}}

\newcommand\bfI{\mathbf{I}}

\newcommand\bfT{\mathbf{T}}

\newcommand\calA{\mathcal{A}}
\newcommand\calB{\mathcal{B}}
\newcommand\calC{\mathcal{C}}
\newcommand\calD{\mathcal{D}}

\newcommand\calF{\mathcal{F}}

\newcommand\calH{\mathcal{H}}

\newcommand\calL{\mathcal{L}}

\newcommand\calP{\mathcal{P}}
\newcommand\calQ{\mathcal{Q}}

\newcommand\calS{\mathcal{S}}

\newcommand\calX{\mathcal{X}}
\newcommand\calZ{\mathcal{Z}}

\newcommand\frakm{\mathfrak{m}}

\newcommand{\Phan}{\operatorname{Phan}}

\def\r{\operatorname{r}}
\def\reg{\operatorname{reg}}

\usepackage{floatrow}
\usepackage{caption}
\DeclareCaptionSubType[alph]{figure}

\begin{document}
\baselineskip=16pt

\title[Regularity and multiplicity]{Regularity and multiplicity of toric rings of three-dimensional Ferrers diagrams}

\author[Kuei-Nuan Lin, Yi-Huang Shen]{Kuei-Nuan Lin* and Yi-Huang Shen}

\thanks{AMS 2020 {\em Mathematics Subject Classification}.
    Primary 13F55, 
    13P10, 
    14M25, 
    16S37; 
    Secondary 14M05, 
    14N10, 
    05E40, 
    05E45} 


\address{Department of Mathematics, The Penn State University, 
McKeesport, PA, 15132, USA}
\email{kul20@psu.edu}

\address{CAS Wu Wen-Tsun Key Laboratory of Mathematics, School of Mathematical Sciences, University of Science and Technology of China, Hefei, Anhui, 230026, P.R.~China}
\email{yhshen@ustc.edu.cn}

\begin{abstract}
   We investigate the Castelnuovo--Mumford regularity and the multiplicity of the toric ring associated with a three-dimensional Ferrers diagram. In particular, in the rectangular case, we provide direct formulas for these two important invariants. Then, we compare these invariants for an accompanying pair of Ferrers diagrams under some mild conditions and bound the Castelnuovo--Mumford regularity for more general cases.
\end{abstract}

\maketitle

\section{Introduction}
In this work, we will consider the Castelnuovo--Mumford regularity and the multiplicity of the toric varieties arising from the squarefree monomial ideals associated with the three-dimensional Ferrers diagrams.

The study of toric varieties is an important part of algebraic geometry.
From our point of view, it suffices to notice that the image of the rational map defined by an equi-generated monomial ideal is a toric variety. Meanwhile, the special fiber of the blowup algebra defined by an equi-generated monomial ideal is the toric variety 
associated with the monomial ideal in this case.  This type of approach has applications in statistics \cite{zbMATH05303649}, geometric modeling \cite{MR2039975}, and coding theory \cite{arXiv:1808.06487}.  In addition, the recent work of Cox, Lin, and Sosa \cite{arXiv:1806.08184} provides a connection of chemical reaction networks with the toric variety arising from the monomial ideal.  

In this paper, we will focus on the two most important invariants of toric varieties: the \emph{Castelnuovo--Mumford regularity} and the \emph{multiplicity} (also known as the \emph{degree}). 

Throughout this paper, $\KK$ is a field of characteristic zero, for simplicity. Recall that, for a finitely generated graded nonzero module $M$ over the polynomial ring $S=\KK[x_1,\dots,x_n]$, the \emph{Castelnuovo--Mumford regularity} of $M$, denoted by $\reg(M)$, is $\max\Set{j-i : \beta_{i,j}\ne 0}$, where the $\beta_{i,j}$'s are the graded Betti numbers of $M$. The regularity can be used to bound the degree of syzygy generators of the module $M$.

Another important invariant that we investigate here is the \emph{multiplicity} $\e(M)$ of $M=\bigoplus_k M_k$ with respect to the graded maximal ideal. Recall that $H(M,k):=\dim_{\KK}(M_k)$ is eventually a polynomial in $k$ of degree $\dim(M)-1$. The leading coefficient of this polynomial is of the form $\e(M)/(\dim(M)-1)!$ for the positive integer $\e(M)$. When $M$ is the homogeneous coordinate ring of a projective variety $X$, the multiplicity is just the degree of $X$. 

These two invariants measure the complexity of the toric variety {from different perspectives}. Finding a reasonable estimate of these two invariants is still wildly open \cite{MR1492542}.  Researchers in commutative algebra, algebraic geometry, geometry modeling, coding theory, and statistics are still working restlessly on investigating these invariants. 

Recently, Biermann, O'Keefe, and Van~Tuyl \cite{MR3595300} found bounds on Castelnuovo--Mumford regularity of toric rings associated with edge ideals arising from complete bipartite graphs. Later, Beyarslan, H\`a, and O’Keefe bound the regularity of toric rings associated with simple graphs using their induced subgraphs in \cite{arXiv:1703.08270}. Consequently, Galetto and his authors \cite{MR4026786} computed the graded Betti numbers for the toric ideal of graphs constructed by adjoining cycles to complete bipartite graphs. The regularity of linearly presented toric edge ideals related to the bipartite complement of a graph is investigated by Greif and McCullough in \cite{MR4123732}. Furthermore, graphs such that their associated toric rings have regularity $r$ and $h$-polynomials have degree $d$ are classified in \cite{MR4143385} by Favacchio and his coauthors.
From this list of recent work, we can see clearly that there are still many unknown cases deserving investigation, even with toric rings associated with  simple graphs.

Meanwhile, Eto gave a method for computing the multiplicity of monoid rings in \cite{MR1948662}.  
Gitler and Valencia \cite{MR2179639} related the multiplicity of the edge subring of a simple graph
with the volume of its edge polytope. 
{At the same time,}
Villarreal \cite[chapter 10]{MR3362802} presented sharp upper bounds for the multiplicity of edge subrings.

In general, not much is known when the toric ring is associated with monomial ideals of degree greater than $2$. One important reason is that finding the explicit defining equations of the toric rings is quite difficult in general. Even when the defining equations are known in some cases, the computation of regularity or multiplicity is still very involved; see, for example, the regularities of toric rings (fiber cones) associated with initial ideals of secant varieties of rational normal scrolls that we determined in \cite{LinShenCollect}.

In this work, we will continue our previous work in \cite{arXiv:1709.03251}. The toric variety that we consider here is associated with the squarefree monomial ideal $I_{\calD}$ where $\calD$ is a three-dimensional Ferrers diagram. The classical {two-dimensional} Ferrers diagram is an important object studied in combinatorics, and it has applications in permutation statistics \cite{MR2095860} and inverse rook problems \cite{MR0429578}. They are also known as Young diagrams and have important applications to the representation theory of symmetric and general linear groups, and to Schubert calculus; see, for example, \cite {MR3438312}. 

The regularity and multiplicity problem for the classical Ferrers diagram has been solved by Corso and Nagel in \cite{MR2457403}.  As a natural generalization, in our previous work \cite{arXiv:1709.03251}, we have shown that, under some mild conditions, the special fiber ideal $J_{\calD}$ associated with a three-dimensional Ferrers diagram $\calD$ has a squarefree quadratic initial ideal with respect to the lexicographic order. In addition, the accompanying Stanley--Reisner complex $\Delta(\calD)$ is pure vertex-decomposable.  The main purpose of the current paper is then to investigate the regularity and multiplicity of the special fiber ring $\calF(I_{\calD})$. Related key ingredients from our previous paper will be summarized here in Section 2. 

By some algebraic argument, we can transfer the calculation of these two vital invariants of the special fiber ring $\calF(I_{\calD})$ to those of the associated Stanley--Reisner ring $\KK[\Delta(\calD)]$; see \Cref{Cor:Fib-SR}. Notice that the vertex-decomposability of the associated Stanley--Reisner complex $\Delta(\calD)$ provides naturally a short exact sequence, on which both invariants behave nicely; see \Cref{rmk:alg}. Therefore, one can quickly extract an algorithm for calculating {the reduction number} of $I_\calD$ as well as the regularity and multiplicity of the special fiber ring $\calF(I_{\calD})$ associated with the three-dimensional Ferrers diagram $\calD$.

When the given three-dimensional Ferrers diagram $\calD=[a]\times[b]\times[c]$ is in a full rectangular shape, we provide a clean formula for {the reduction number,} the regularity, and the multiplicity of the special fiber ring in Section 4.  On the other hand, we do not intend to give a closed formula for
them in the general case. Even for the multiplicity in the two-dimensional case, the formula given in \cite[Corollary 5.6]{MR2457403} is already highly entangled. Therefore, the main object of this paper is to give combinatorically reasonable upper bounds for these two invariants.

We exemplify two approaches.  In Section 5, we consider an accompanying pair of Ferrers diagrams $\calD_1\subseteq \calD_2$. To compare these invariants on these two diagrams, we fall back on the aforementioned shedding order of the associated vertex-decomposable Stanley--Reisner complexes $\Delta(\calD_1)$ and $\Delta(\calD_2)$. Since this approach requires a synchronizing comparison along the decomposing process, we need a slightly stronger condition, given in \Cref{def:strongPP}.
In the second approach, an identical estimate can be achieved for {the reduction number and} the regularity under the same (weaker) condition as assumed in \cite{arXiv:1709.03251}. No doubt, the proof, given in Section 6, invites a more intricate combinatorial maneuver.

\section{Preliminaries}
Throughout this paper, when $n$ is a non-negative integer, we will follow the common convention and denote the set $\{1,2,\dots,n\}$ simply as $[n]$.

Let $\calD$ be a nonempty set of finite lattice points in $\ZZ_+^3$. Let 
\[
    m\coloneqq \max\Set{i:(i,j,k)\in\calD}, \quad
    n\coloneqq \max\Set{j:(i,j,k)\in\calD} \quad \text{and} \quad
    p\coloneqq \max\Set{k:(i,j,k)\in\calD}.
\]  
We associate with  $\calD$ the polynomial ring 
\[
    R=\KK[x_1,\dots,x_m,y_1,\dots,y_n,z_1,\dots,z_p]
\]
and the monomial ideal 
\[
    I_{\calD}\coloneqq (x_iy_jz_k:(i,j,k)\in \calD)\subset R.
\]
This ideal will be called the \emph{defining ideal} of $\calD$.

If we write $\frakm$ for the graded maximal ideal of $R$, the \emph{special
    fiber ring} of $I_{\calD}$ is
\[
    \calF(I_{\calD})\coloneqq \bigoplus_{l\ge 0} I_{\calD}^{l}/\frakm I_{\calD}^l \cong
    R[I_{\calD}t] \otimes_{R} R/\frakm.
\]
Sometimes, we also call it the \emph{toric ring} of $I_{\calD}$ and denote it by $\KK[I_{\calD}]$. Let 
\[
    \KK[\bfT_{\calD}]\coloneqq \KK[T_{i,j,k}: (i,j,k)\in \calD]
\]
be the polynomial ring in the variables $T_{i,j,k}$ over the field $\KK$.  Consider the map $\varphi:\KK[\bfT_{\calD}]\to R$, given by $T_{i,j,k}\to x_iy_jz_k$, and extend algebraically. Then $\calF(I_{\calD})$ is canonically isomorphic to $\KK[\bfT_{\calD}]/\ker(\varphi)$. We will denote the kernel ideal $\ker(\varphi)$ by $J_{\calD}$ and call it the \emph{special fiber ideal} of $I_{\calD}$. Sometimes, we also call it the \emph{toric ideal} of $I_{\calD}$. It is well-known that $J_{\calD}$ is a graded binomial ideal; see, for instance, \cite[Corollary 4.3]{MR1363949} or \cite{MR2611561}.  We also observe that $\calF(I_{\calD})$, being isomorphic to a subring of $R$, is a domain. Hence $J_{\calD}$ is a prime ideal.

\begin{Definition}
    \label{3F}
    Let $\calD\ne \varnothing$ be a finite {set} of lattice points in $\ZZ_+^3$. 
    \begin{enumerate}[a]
        \item 
            We will call 
            \[
                a_{\calD}\coloneqq \left|\Set{i:(i,j,k)\in\calD}\right|, \quad
                b_{\calD}\coloneqq \left|\Set{j:(i,j,k)\in\calD}\right| \text{ and }
                c_{\calD}\coloneqq \left|\Set{k:(i,j,k)\in\calD}\right|
            \]
            the \emph{essential length, width, and height} of $\calD$ respectively. 
        \item The set $\calD$ is called a \emph{three-dimensional Ferrers diagram} if for each $(i_0,j_0,k_0)\in \calD$, and for every positive integers $i\le i_0$, $j\le j_{0}$ and $k\le k_0$, one has $(i,j,k)\in \calD$.
    \end{enumerate}
\end{Definition}

\begin{Definition} 
    Let $\calD$ be a three-dimensional Ferrers diagram.  For each $i \in [a_{\calD}]$, let \\$b_i=\max\Set{j: (i,j,1)\in \calD}$ and $c_i=\max\Set{k: (i,1,k)\in \calD}$. Then the \emph{projection} of the $x=i+1$ layer is the set
    \[
        \Set{(i,j,k)\in \ZZ_+^3: j\le b_{i+1} \text{ and }k\le c_{i+1}}
    \]
    when $i<a_\calD$.
    And $\calD$ is said to satisfy the \emph{projection property} if the $x=i$ layer covers the projection of the $x=i+1$ layer for each $i \in [a_{\calD}-1]$, i.e., the following equivalent conditions hold:
    \begin{enumerate}[a]
        \item $(i,b_{i+1},c_{i+1})\in \calD$;
        \item if $(i+1,j_1,k_1)\in \calD$ and $(i+1,j_2,k_2)\in \calD$, then $(i,j_1,k_2)\in \calD$.
    \end{enumerate}
\end{Definition}

\begin{Definition}
    \label{2-minors}
    Let $\calD\ne \varnothing$ be a finite {set} of lattice points in $\ZZ_+^3$.  For
    $\bdu=(i_1,j_1,k_1)$ and $\bdv=(i_2,j_2,k_2)$ in $\calD$, define 
    \[
        \bfI_{2,x}(\bdu,\bdv)\coloneqq 
        \begin{cases}
            T_{\bdu}T_{\bdv}-T_{i_2,j_1,k_1}T_{i_1,j_2,k_2}, & \text{if
                $(i_2,j_1,k_1), (i_1,j_2,k_2)\in \calD$},\\
            0, & \text{otherwise}.
        \end{cases}
    \]
    When this is a nonzero binomial, we will say \emph{switching the $x$-coordinates} is allowed between $\bdu$ and $\bdv$.
    We can similarly define $\bfI_{2,y}$ and $\bfI_{2,z}$. Now, let
    \[
        \bfI_2(\calD)\coloneqq \Big(\bfI_{2,x}(\bdu,\bdv),\bfI_{2,y}(\bdu,\bdv),
        \bfI_{2,z}(\bdu,\bdv): \bdu, \bdv\in \calD\Big)\subset
        \KK[\bfT_{\calD}],
    \]
    and call it the \emph{$2$-minors ideal} of $\calD$.  
\end{Definition}

\begin{Example}
    \label{MinExam}
    Let $\calD$ be a three-dimensional Ferrers diagram 
    consisting
    the following lattice points 
    \begin{align*}
        (1,1,1),
        (1,1,2), 
        (1,1,3), 
        (1,2,1), 
        (1,2,2), 
        (1,2,3), 
        (1,3,1), 
        (1,3,2),
        (1,3,3),\\ 
        (2,1,1), 
        (2,1,2), 
        (2,1,3), 
        (2,2,1),
        (2,2,2),
        (2,3,1), 
        (3,1,1), 
        (3,1,2), 
        (3,2,1). 
    \end{align*}
    This diagram satisfies the projection property by easy verification. 
    We can use the base ring 
    \begin{align*}
        S&=\KK[
        T_{1,1,1},
        T_{1,1,2}, 
        T_{1,1,3}, 
        T_{1,2,1}, 
        T_{1,2,2}, 
        T_{1,2,3}, 
        T_{1,3,1}, 
        T_{1,3,2},
        T_{1,3,3},\\
        & \qquad\qquad  T_{2,1,1}, 
        T_{2,1,2}, 
        T_{2,1,3}, 
        T_{2,2,1},
        T_{2,2,2},
        T_{2,3,1}, 
        T_{3,1,1}, 
        T_{3,1,2}, 
        T_{3,2,1}], 
    \end{align*}
    and apply the canonical epimorphism. The toric ideal $J_{\calD}$ in $S$ contains, among others, the binomials $ T_{2,1,2}T_{3,2,1}-T_{3,1,2}T_{2,2,1}$ \textup{(}an $x$-coordinate switch\textup{)}, $ T_{1,2,1}T_{1,3,2}-T_{1,3,1}T_{1,2,2}$ \textup{(}a $y$-coordinate switch\textup{)}, and $ T_{1,1,1}T_{1,2,2}-T_{1,1,2}T_{1,2,1}$ \textup{(}a $z$-coordinate switch\textup{)}. 
\end{Example}

If $\calD$ is a three-dimensional Ferrers diagram that satisfies the projection property, it is proved in \cite[Corollary 2.14 and Theorem 6.1]{arXiv:1709.03251} that the toric ideal $J_{\calD}$ coincides with the $2$-minors ideal $\bfI_{2}(\calD)\subseteq \KK[\bfT_{\calD}]$. Furthermore, if we order the $T$-variables lexicographically with respect to their subscripts and apply the lexicographic monomial order to $\KK[\bfT_{\calD}]$ (we will call it  \emph{the lexicographic order} for short), the $2$-minors in $\bfI_{2}(\calD)$ provide a minimal Gr\"obner basis of $J_{\calD}$. On the other hand,  for arbitrary integer $p\ge 4$, we can consider the three-dimensional Ferrers diagram $\calD$ minimally generated (governed) by the following extremal lattices points:
\begin{gather*}
    (1,2,p-1),(2,3,p-2),(3,4,p-3),\dots(p-1,p,1),(p,1,\underline{2}),\\
    (2,1,p-1),(3,2,p-2),(4,3,p-3),\dots,(p,p-1,1),(1,p,\underline{2}).
\end{gather*}
Then, we can actually prove that the minimal generating set of the special fiber ideal $J_D$ contains a degree $p$ binomial, generalizing the phenomenon observed in \cite[Example 2.4]{arXiv:1709.03251} for degree three.

As the main result of \cite{arXiv:1709.03251}, we proved that the special fiber ring $\calF(I_{\calD})$ is a Koszul Cohen--Macaulay normal domain, provided that $\calD$ is a three-dimensional Ferrers diagram that satisfies the projection property. Its proof is quite involved and depends on the following two key ingredients.
\begin{enumerate}[a]
    \item For both the Cohen--Macaulayness and the primeness of the ideal $\bfI_2(\calD)$, we had to use a common elaborate induction process, which will be explained afterward.
    \item To establish the Cohen--Macaulayness of $\bfI_2(\calD)$, we actually proved that the Stanley--Reisner complex $\Delta(\calD)$ associated with  the squarefree initial ideal $\ini(\bfI_2(\calD))$ is pure vertex-decomposable. Furthermore, the above induction process induces a shedding order for that purpose.
\end{enumerate}

To introduce the aforementioned induction process, we still need some preparation.

\begin{Definition}  
    \label{def:alpha-beta-gamma}
    For each $\bdu=(i_0,j_0,k_0)\in \calD$, let
    \[
        \alpha_{\calD}(\bdu)\coloneqq \max\Set{i: (i,j_0,k_0)\in \calD}.
    \]
    In a similar vein, we can define $\beta_{\calD}(\bdu)$ and $\gamma_{\calD}(\bdu)$. Meanwhile, we use the superscript to denote the corresponding $x$ layers. For instance,
    \[
        \calD^1\coloneqq \Set{(1,j,k)\in \calD} \quad \text{ and } \quad
        \calD^{\ge d}\coloneqq \Set{(i,j,k)\in \calD: i\ge d}.
    \] 
\end{Definition}

It is clear that when $\bdu=(1,1,1)$ and $\calD$ is a three-dimensional Ferrers diagram, then $\alpha_{\calD}(\bdu)$, $\beta_{\calD}(\bdu)$, and $\gamma_{\calD}(\bdu)$ are $a_\calD$, $b_\calD$, and $c_\calD$ respectively; see also \Cref{3F}.

\begin{Definition}
    \label{def:zones}
    Let $\calD$ be a three-dimensional Ferrers diagram.  Take arbitrary $\bdu=(i_0,j_0,k_0)\in \calD$ and for simplicity, write 
    \[
        \alpha=\alpha_{\calD}(\bdu),\quad 
        \beta=\beta_{\calD}(\bdu)\quad \text{and} \quad
        \gamma=\gamma_{\calD}(\bdu).
    \]
    Then we can divide $\calD^{\ge i_0}$ into the following six zones:
    \begin{align*}
        \calZ_1(\calD,\bdu)\coloneqq &\Set{(i,j,k)\in \calD^{\ge i_0} : 1\le j\le j_0 \text{ and } k>\gamma},\\
        \calZ_2(\calD,\bdu)\coloneqq &\Set{(i,j,k)\in \calD^{\ge i_0} : 1\le j\le j_0 \text{ and }k_0<k\le \gamma},\\ 
        \calZ_3(\calD,\bdu)\coloneqq &\Set{(i,j,k)\in \calD^{\ge i_0} : 1\le j\le j_0 \text{ and }1\le k\le k_0},\\ 
        \calZ_4(\calD,\bdu)\coloneqq &\Set{(i,j,k)\in \calD^{\ge i_0} : j_0< j\le \beta \text{ and }k_0<k\le \gamma},\\ 
        \calZ_5(\calD,\bdu)\coloneqq &\Set{(i,j,k)\in \calD^{\ge i_0} : j_0< j\le \beta \text{ and }1\le k\le k_0},\\ 
        \calZ_6(\calD,\bdu)\coloneqq &\Set{(i,j,k)\in \calD^{\ge i_0} : j>\beta \text{ and }1\le k< k_0}.  
    \end{align*}
    It is clear that $\calD^{\ge i_0}$ is the disjoint union of the above six zones. In the subsequent discussion, they will be called the $\calZ$-zones with respect to $\calD$ and $\bdu$. We will omit some of the parameters if they are clear from the context. \Cref{Zone} gives the idea of the division of $\calD^{i_0}$ with respect to these zones.

    \begin{figure}[htb] 
        \begin{center}
            \scalebox{1.4}{
                \begin{tikzpicture}[thick, scale=0.5, every node/.style={scale=0.6}]
                    \draw (7,0) node [below right]{$y$};
                    \draw (0,6) node [above left]{$z$};

                    \draw [pattern=crosshatch] (2.5,2.5)--(2,2.5)--(2,2)--(2.5,2)--cycle; 
                    \node [above right] at (4,5) {$(j_0,k_0)$}; 
                    \draw [->, thin] (2.25,2.25) to [bend left] (4,5);

                    \path (0,0)--(0,4)  node [below left]{$\gamma$};
                    \path (0,-0.2)--(5,-0.2) node [below left]{$\beta$};

                    \draw (0,4)--(2,4)--(2,5)--(1.5,5)--(1.5,5.5)--(0.5,5.5)--(0.5,6)--(0,6)--cycle;
                    \node [right] at (0.5,4.7) {$\mathcal{Z}_1$};

                    \draw (0,2.5)--(2.5,2.5)--(2.5,4)--(0,4)--cycle;
                    \node [right] at (0.5,3.3) {$\mathcal{Z}_2$};

                    \draw (0,0)--(2.5,0)--(2.5,2.5)--(0,2.5)--cycle;
                    \node [right] at (0.5,1.2) {$\mathcal{Z}_3$};

                    \draw (2.5,2.5)--(5,2.5)--(5,3.5)--(4,3.5)--(4,4)--(2.5,4)--cycle;
                    \node [right] at (3.2,3.3) {$\mathcal{Z}_4$};

                    \draw (2.5,0)--(5,0)--(5,2.5)--(2.5,2.5)--cycle;
                    \node [right] at (3.2,1.2) {$\mathcal{Z}_5$};

                    \draw (5,0)--(7,0)--(7,1)--(6.5,1)--(6.5,2)--(5,2)--cycle;
                    \draw (5.9,1.2) node {$\mathcal{Z}_6$};
                \end{tikzpicture}}
        \end{center}
        \caption{$\calZ$-zones with respect to $\calD$ and $\bdu$}\label{Zone}
    \end{figure}
\end{Definition}

\begin{Definition}
    Let $\calD$ be a nonemtpy subset of $\ZZ_+^3$. We say that a total order $\prec$ on $\calD$ is a \emph{quasi-lexicographic order} if it satisfies the following two conditions. 
    \begin{enumerate}[a]
        \item 
            The points in $\calD^i$ precede the points in $\calD^{i+1}$ with respect to $\prec$ for each $i\in[a_{\calD}-1]$.
        \item  
            For distinct $\bdu=(i,j_1,k_1)$ and $\bdv=(i,j_2,k_2)$ in $\calD^{i}$, if $j_1\le j_2$ and $k_1\le k_2$, then $\bdu$ precedes $\bdv$ with respect to $\prec$. 
    \end{enumerate}
\end{Definition}

Obviously, the common lexicographic order is a quasi-lexicographic order. 

\begin{Definition}
    \label{def:initial-part}
    Let $\prec$ be a quasi-lexicographic order on a nonemtpy subset $\calD$ of $\ZZ_+^3$.
    Given a lattice point $\bdu\in \calD$, let $\calA_{\bdu}$ be the subset obtained from $\calD$ by removing the points before $\bdu$ with respect to $\prec$. We also write $\calA_{\bdu}^{+}\coloneqq \calA_{\bdu}\setminus \bdu$. 
\end{Definition}

\begin{Setting}
    [Induction Order]
    \label{set:IO}
    Suppose that $\calD$ is a three-dimensional Ferrers diagram with $a_{\calD}\ge 2$. We adopted in \cite{arXiv:1709.03251} the following total order $\prec_\calD$ on $\calD^1$ for considering both the vertex-decomposability of the associated Stanley--Reisner complex $\Delta(\calD)$ and the primeness of $\bfI_2(\calD)$.
    By abuse of terminology, we call this total order the \emph{induction order} with respect to $\calD$. 

    Let $\calC_1=\calC_1(\calD)\coloneqq \Set{(1,j,k)\in \calD: k\le c_{\calD^{\ge 2}}}$ be the points in the \emph{first stage}, and $\calC_2=\calC_2(\calD)\coloneqq \calD^1\setminus \calC_1$ be the points in the \emph{second stage}. For $\prec_\calD$, we require the following.
    \begin{enumerate}[a]
        \item The points in the first stage precede the points in the second stage with respect to $\prec_\calD$.
        \item The restriction of $\prec_\calD$ to $\calC_1$ is the lexicographic order.
        \item Consider the symmetry operation $\calS: \ZZ_+^3\to \ZZ_+^3$ by sending $(i,j,k)$ to $(i,k,j)$.  We will also call this operation as a \emph{flip}. Then the restriction of $\prec_\calD$ to $\calC_2$ corresponds to the lexicographic order on $\calS(\calC_2)$.
    \end{enumerate}
    The total order $\prec_\calD$ on $\calD^1$ is completely determined by the three requirements just stated. {We may also extend this order to the whole $\calD$ by arranging the points in $\calD^{\ge 2}$ in a similar fashion and putting them after the points in $\calD^1$. 
        The induction order $\prec_\calD$ constructed by this approach is a quasi-lexicographic order.}

    Figure \ref{Fig:InductionOrder} gives an idea of how this proceeds on the $x=1$ layer. The first point is $\circ=(1,1,1)$. When $c_{\calD^{\ge 2}}<c_{\calD}$, the last point is $\bullet=(1,\beta_{\calD}( (1,1,c_{\calD})),c_{\calD})$ of $\calD$, using the notation we introduced in \Cref{def:alpha-beta-gamma}. Otherwise, $c_{\calD^{\ge 2}}=c_{\calD}$ with the second stage disappears and the last point is $(1,b_{\calD},\gamma_{\calD}( (1,b_{\calD},1)))$ of $\calD$.  Meanwhile, in Figure \ref{Fig:InductionOrder}, $\varstar$ denotes the last point in the first stage while $\square$ denotes the first point in the second stage.  

    \begin{figure}[h] 
        \begin{center}
            \scalebox{1.4}{
                \begin{tikzpicture}[scale=0.6, every node/.style={scale=0.6}] 

                    \fill [gray!30] (0,3.5)--(4,3.5)--(4,4)--(0,4)--cycle; 

                    \draw [line width=1.2pt]  (0,0)--(7,0)--(7,1)--(6.5,1)--(6.5,2)--(5,2)--(5,2.5)--(5,3.5)--(4,3.5)--(4,4)--(2.5,4)--(2,4)--(2,5)--(1.5,5)--(1.5,5.5)--(1,5.5)--(1,6)--(0,6)--cycle;

                    \coordinate[label=right:{$y$}] (y) at (7,0);
                    \coordinate[label=above:{$z$}] (z) at (0,6);
                    \coordinate[label=below left:{$c_{\mathcal{D}^{\ge 2}}$}] (gamma) at (0,4);
                    \coordinate[label={$\circ$}] (O1) at (0.25,0.05);
                    \coordinate[label={$\varstar$}] (T2) at (6.75,0.5);
                    \coordinate[label={$\bullet$}] (T1) at (0.75,5.5);
                    \coordinate[label={$\square$}] (O2) at (0.25,3.99);

                    \draw [->] (0.25,0.45) -- (0.25,3.75);
                    \draw[ thin,  dotted,->-=.8] (0.25,3.75)--(0.5,3.75)--(0.5,0.25)--(0.75,0.25);
                    \draw [->] (0.75,0.25) -- (0.75,3.75);
                    \draw[ thin,  dotted,->-=.8] (0.75,3.75)--(1,3.75)--(1,0.25)--(1.25,0.25);
                    \draw [->] (1.25,0.25) -- (1.25,3.75);
                    \draw[ thin,  dotted,->-=.8] (1.25,3.75)--(1.5,3.75)--(1.5,0.25)--(1.75,0.25);
                    \draw [->] (1.75,0.25) -- (1.75,3.75);
                    \draw[ thin,  dotted,->-=.8] (1.75,3.75)--(2,3.75)--(2,0.25)--(2.25,0.25);
                    \draw [->] (2.25,0.25) -- (2.25,3.75);
                    \draw[ thin,  dotted,->-=.8] (2.25,3.75)--(2.5,3.75)--(2.5,0.25)--(2.75,0.25);
                    \draw [->] (2.75,0.25) -- (2.75,3.75);
                    \draw[ thin,  dotted,->-=.8] (2.75,3.75)--(3,3.75)--(3,0.25)--(3.25,0.25);
                    \draw [->] (3.25,0.25) -- (3.25,3.75);
                    \draw[ thin,  dotted,->-=.8] (3.25,3.75)--(3.5,3.75)--(3.5,0.25)--(3.75,0.25);
                    \draw [->] (3.75,0.25) -- (3.75,3.75);
                    \draw[ thin,  dotted,->-=.8] (3.75,3.75)--(4,3.75)--(4,0.25)--(4.25,0.25);
                    \draw [->] (4.25,0.25) -- (4.25,3.25);
                    \draw[ thin,  dotted,->-=.8] (4.25,3.25)--(4.5,3.25)--(4.5,0.25)--(4.75,0.25);
                    \draw [->] (4.75,0.25) -- (4.75,3.25);
                    \draw[ thin,  dotted,->-=.8] (4.75,3.25)--(5,3.25)--(5,0.25)--(5.25,0.25);
                    \draw [->] (5.25,0.25) -- (5.25,1.75);
                    \draw[ thin,  dotted,->-=.8] (5.25,1.75)--(5.5,1.75)--(5.5,0.25)--(5.75,0.25);
                    \draw [->] (5.75,0.25) -- (5.75,1.75);
                    \draw[ thin,  dotted,->-=.8] (5.75,1.75)--(6,1.75)--(6,0.25)--(6.25,0.25);
                    \draw [->] (6.25,0.25) -- (6.25,1.75);
                    \draw[ thin,  dotted,->-=.8] (6.25,1.75)--(6.5,1.75)--(6.5,0.25)--(6.75,0.25);
                    \draw [->] (6.75,0.25) -- (6.75,0.55);

                    \draw[->] (0.4,4.25)--(1.75,4.25);
                    \draw[thin, dotted,->-=.8] (1.75,4.25)--(1.75,4.5)--(0.25,4.5)--(0.25,4.75);
                    \draw[->] (0.25,4.75)--(1.75,4.75);
                    \draw[thin, dotted,->-=.8] (1.75,4.75)--(1.75,5)--(0.25,5)--(0.25,5.25);
                    \draw[->] (0.25,5.25)--(1.25,5.25);
                    \draw[thin, dotted] (1.25,5.25)--(1.25,5.5)--(0.25,5.5)--(0.25,5.75);
                    \draw[->](0.25,5.75)--(0.6,5.75);

                    \draw [decorate,decoration={brace, mirror}] (7.5,0.25) - - node[right] {\ 1st stage} (7.5,3.75);
                    \draw [decorate,decoration={brace, mirror}] (7.5,4.25) - - node[right] {\ 2nd stage} (7.5,5.75);
                \end{tikzpicture}
            }
        \end{center}
        \caption{Induction Order}\label{Fig:InductionOrder}
    \end{figure}   
\end{Setting}

\begin{Definition}
    \label{def:essentially}
    Let $\calD$ be a nonemtpy subset of $\ZZ_+^3$.  
    \begin{enumerate}[a]
        \item  Suppose that $\calD^{i_0}=\varnothing$
            for some $i_0\in\ZZ_{+}$. Then, we can remove the whole $x=i_0$ layer from the ambient space, and consider a new set of lattice points 
            \[
                \calD':=\Set{(i,j,k)\in \calD: i<i_0}\cup \Set{(i-1,j,k): (i,j,k)\in \calD\text{ with }i>i_0}.
            \]
            Deriving $\calD'$ from $\calD$ above will be called a \emph{reduction along the $x$-direction}. It is not difficult to see that $\calF(I_\calD)\cong \calF(I_{\calD'})$.
        \item One can similarly define reductions along the $y$-direction and $z$-direction. In this paper, when we say a set $\calD$ \emph{essentially} satisfies some property $(P)$, we mean that after a finite sequence of reductions, the derived new set satisfies the property $(P)$.
    \end{enumerate}
\end{Definition}

\begin{Example}
    For example, we may have 
    \[
        \calD=\Set{(1,1,1),(1,1,2),(1,2,1),(3,1,1),(3,1,2)}.
    \]
    Since $\calD^2=\varnothing$, we can apply a reduction along the $x$-direction to get \[
        \calD'=\Set{(1,1,1),(1,1,2),(1,2,1),(2,1,1),(2,1,2)}
    \]
    with respect to $i_0=2$.
    The essential length, width and height do not change: $a_\calD=a_{\calD'}$, $b_\calD=b_{\calD'}$ and  $c_\calD=c_{\calD'}$.
    But visually, this $\calD'$ is more concise than the original $\calD$.
\end{Example}

\section{Considering the Hilbert polynomials}
Let $M\ne 0$ be a finitely generated graded module over the polynomial ring $S=\KK[x_1,\dots,x_n]$. 
For simplicity, the field $\KK$ has characteristic zero. Let $\Hilb_M(t)\coloneqq \sum_{k\in \ZZ}H(M,k)t^{k}$ be the Hilbert series of $M$. It is well-known that 
\[
    \Hilb_M(t)=P_M(t)/(1-t)^d \quad \text{with}\quad P_M(1)\ne 0
\]
for some \emph{$h$-polynomial} $P_M(t)\in \ZZ[t,t^{-1}]$.
The number $d$ is the Krull dimension of $M$ while $P_M(1)=\e(M)>0$ is the \emph{multiplicity} of $M$;  see \cite[Section 6.1.1]{MR2724673}.

\begin{Lemma}
    [{\cite[Proposition 7.43]{MR2508056}}]
    Let $M$ be a graded Cohen-Macaulay module of dimension $d$ over the polynomial ring $S$. Let $P_M(t)$ be the $h$-polynomial of the Hilbert series of $M$. Then $\reg(M)=\deg(P_M(t))$.
\end{Lemma}

\begin{Lemma}
    \label{reg-ini-reg}
    Let $J$ be a homogeneous ideal of $S$. If with respect to some monomial order of $S$, the initial ideal $\ini(J)$ is Cohen--Macaulay, then $\reg(S/J)=\reg(S/\ini(J))$.  
\end{Lemma}

\begin{proof}
    It is well-known that $\Hilb_{S/J}(t)$ coincides with $\Hilb_{S/\ini(J)}(t)$; see \cite[Corollary 6.1.5]{MR2724673}.  Furthermore, since $\ini(J)$ is Cohen--Macaulay, so is $J$ of the same dimension by \cite[Theorem 3.3.4 and Corollary 3.3.5]{MR2724673}.  Therefore, 
    \[
        \reg(S/\ini(J))=\deg(P_{S/\ini(J)}(t))=\deg(P_{S/J}(t))=\reg(S/J). \qedhere
    \]
\end{proof}

Let $I$ be an ideal.  A subideal $J\subseteq I$ is called a \emph{reduction} of $I$ if there is a number $n$ such that $I^{n+1}=J\cdot I^{n}$. The least number $n$ with the above property is the \emph{reduction number} of $I$ with respect to $J$ and denoted by $r_J(I)$. A reduction $J$ is \emph{minimal} if no proper subideal of $J$ is a reduction of $I$. The \emph{\textup{(}absolute\textup{)} reduction number} of $I$ is defined as
\[
    \r(I)\coloneqq \min\Set{\r_J(I)| \text{$J$ is a minimal reduction of $I$}}.
\]

\begin{Lemma}
    [{\cite[Proposition 6.6]{CNPY}}]
    \label{CNPY:6.6}
    Let $I \subset S$ be a homogeneous ideal that is generated in one degree, say {$\delta$}. Assume that the special fiber ring $\calF(I)$ is Cohen-Macaulay. Then each minimal reduction of $I$ is generated by $\dim(\calF(I))$ homogeneous polynomials of degree {$\delta$}, and $I$ has reduction number $\r(I) = \reg(\calF(I))$.
\end{Lemma} 

Let $\calD$ be a nonempty finite subset of $\ZZ_+^3$ such that for the $2$-minors ideal $\bfI_2(\calD)$,  all its minimal Gr\"obner basis elements with respect to the lexicographic order are quadratic (this requirement is satisfied when $\calD$ is a three-dimensional Ferrers diagram that satisfies the projection property, by \cite[Corollary 2.14]{arXiv:1709.03251}). In particular, $\ini(\bfI_2(\calD))$ is squarefree by the well-known Buchberger's criterion and the explicit description of the generating set of $\bfI_2(\calD)$ in \Cref{2-minors}. Whence, we will write $\Delta(\calD)$ for the associated Stanley--Reisner complex.

Once we have a simplicial complex $\Delta$, we will consider the regularity 
\[
    \reg(\Delta)\coloneqq \reg(\KK[\Delta])=\reg(S/I_{\Delta}) 
\]
for the Stanley--Reisner ideal $I_{\Delta}$ of $\Delta$ in appropriate polynomial ring $S$ over the field $\KK$.  We will similarly consider the multiplicity 
\[
    \e(\Delta)\coloneqq \e(\KK[\Delta])=\e(S/I_{\Delta}).
\]

\begin{Corollary}
    \label{Cor:Fib-SR}
    Let $\calD$ be a three-dimensional Ferrers diagram that satisfies the projection property. Then,
    \begin{equation}
        \reg(\Delta(\calD))=\reg(\KK[\bfT_{\calD}]/J_{\calD})=\reg(\calF(I_{\calD}))=\r(I_{\calD})  
        \label{reg=rednum}
    \end{equation}
    and
    \begin{equation}
        \e(\Delta(\calD))=\e(\KK[\bfT_{\calD}]/J_{\calD})=\e(\calF(I_{\calD})).  
        \label{e-equality}
    \end{equation}
\end{Corollary}

\begin{proof}
    With respect to \eqref{reg=rednum}, the first equality follows from \Cref{reg-ini-reg}, and the last equality is by \Cref{CNPY:6.6}. With respect to \eqref{e-equality}, the first equality follows from the fact that the Hilbert series of $\KK[\bfT_{\calD}]/J_{\calD}$ and $\KK[\bfT_{\calD}]/\ini(J_{\calD})$ coincide, as already mentioned in the proof of \Cref{reg-ini-reg}.
\end{proof}

\begin{Remark}
    It is recently proved in \cite[Corollary 2.7]{arXiv1805.11923} that if $I$ is a homogeneous ideal of $S$ with arbitrary term order such that the initial ideal $\ini(I)$ is squarefree, then $\depth(S/I)=\depth(S/\ini(I))$  and $\reg(S/I)=\reg(S/\ini(I))$. In particular, the first equality of \eqref{reg=rednum} also follows.
\end{Remark}

In \cite[Theorem 4.1]{arXiv:1709.03251}, we have shown that if the three-dimensional Ferrers diagram satisfies the projection property, then the associated Stanley--Reisner complex is pure vertex-decomposable. To investigate the regularity and the multiplicity of the corresponding Stanley--Reisner ring, we will fall back on the following critical observation.

\begin{Remark}
    [{See also \cite[Remark 2.4]{MR2426505}}]
    \label{rmk:alg}
    Let $\Delta$ be a pure vertex-decomposable simplicial complex on the finite set $[n]$ and assume that $n$ is a shedding vertex.  Let $I_{\Delta}$ be the Stanley--Reisner ideal of $\Delta$ considered as a complex on $[n]$ in $S=\KK[x_1,\dots,x_n]$. Then the cone over $\link_{\Delta}(n)$ with apex $n$ considered as a complex on $[n]$ has Stanley--Reisner ideal $J_{\link_{\Delta}(n)}=I_{\Delta}:x_n$.  And the Stanley--Reisner ideal of ${\Delta\setminus n}$ considered as a complex on $[n]$ is $(x_n,I_{\Delta\setminus n})$ where $I_{\Delta\setminus n}\subset \KK[x_1,\dots,x_{n-1}]$ is the Stanley--Reisner ideal of $\Delta\setminus n$ considered as a complex on $[n-1]$.  Furthermore, we have a short exact sequence of graded $S$-modules of the same positive dimension: 
    \[
        0\to S/J_{\link_{\Delta}(n)}(-1)\to S/I_{\Delta}\to S/I_{\Delta\setminus n}S\to 0.
    \]
    As multiplicity is additive on such a sequence, we have
    \begin{equation}
        \e(\Delta)=\e(\Delta\setminus n)+\e(\link_{\Delta}(n)). \label{e-seq}
    \end{equation}
    Meanwhile, by \cite[Theorem 4.2]{arXiv:1301.6779}, we have
    \begin{equation}
        \reg(\Delta)=\max\Set{\reg(\Delta\setminus n),\reg(\link_{\Delta}(n))+1}. \label{reg-seq}
    \end{equation}
\end{Remark}

\section{Full rectangular case}
In this section, we will focus on the special case when $\calD=[a_{\calD}]\times [b_{\calD}]\times[c_{\calD}]$ is a full three-dimensional Ferrers diagram.
In this situation, the most convenient tool will be the Segre product of graded modules. Say, that $R=\KK[x_1,\dots,x_m]$ and $S=\KK[y_1,\dots,y_n]$ are two standard graded polynomial rings over $\KK$. Then the \emph{Segre product} of $R$ and $S$ is $R\uotimes S=\bigoplus_{\ell\in \ZZ}(R_{\ell}\otimes_{\KK}S_{\ell})$, which is a graded ring. For a graded $R$-module $M$ and a graded $S$-module $N$, the \emph{Segre product} of $M$ and $N$ is defined as $M\uotimes N= \bigoplus_{\ell\in\ZZ}(M_{\ell}\otimes_{\KK}N_{\ell})$, which is a graded $(R\uotimes S)$-module.

Now, we study the special fiber ring in the full rectangular case. Firstly, we consider multiplicity.
\begin{Lemma}
    \label{lem:14}
    If $M$ and $N$ above are finitely generated and have positive dimensions, then
    \[
        \dim(M\uotimes N)=\dim(M)+\dim(N)-1
    \]
    and
    \[
        \e(M\uotimes N)=\binom{\dim(M)+\dim(N)-2}{\dim(M)-1} \e(M)\e(N).
    \]
\end{Lemma}

\begin{proof}
    By definition, the Hilbert functions satisfy
    \begin{align*}
        H(M,t)&=\frac{\e(M)}{(\dim(M)-1)!} t^{\dim(M)-1}+\text{lower degrees},\\
        \intertext{and}
        H(N,t)&=\frac{\e(N)}{(\dim(N)-1)!} t^{\dim(N)-1}+\text{lower degrees}
    \end{align*}
    for $t\gg 0$.  Thus,
    \begin{align*}
        H(M\uotimes N,t)&=H(M,t)H(N,t)\\
        &= \frac{\e(M)}{(\dim(M)-1)!} \frac{\e(N)}{(\dim(N)-1)!} t^{\dim(M)+\dim(N)-2}+\text{lower degrees}
    \end{align*}
    for $t\gg 0$. The expected dimension and the multiplicity formula can be read off from the last equation. 
\end{proof}

\begin{Proposition}
    \label{full-e}
    Suppose that $\calD$ is the full three-dimensional Ferrers diagram $[a_{\calD}]\times [b_{\calD}]\times[c_{\calD}]$. Then the multiplicity of the special fiber ring is given by the trinomial:
    \[
        \e(\calF(I_{\calD}))= \binom{a_{\calD}+b_{\calD}+c_{\calD}-3}{a_{\calD}-1,b_{\calD}-1,c_{\calD}-1} := \frac{(a_{\calD}+b_{\calD}+c_{\calD}-3)!}{(a_{\calD}-1)!(b_{\calD}-1)!(c_{\calD}-1)!}.
    \]
\end{Proposition}

\begin{proof}
    Notice that
    \begin{align*}
        \calF(I_{\calD}) & \cong \KK[x_1,\dots,x_{a_{\calD}}] \uotimes(\KK[y_1,\dots,y_{b_{\calD}}] \uotimes \KK[z_{1},\dots,z_{c_{\calD}}]).
    \end{align*}
    Thus, by \Cref{lem:14},
    \begin{align*}
        \e(\calF(I_{\calD}))=\binom{a_{\calD}+b_{\calD}+c_{\calD}-3}{a_{\calD}-1}\binom{b_{\calD}+c_{\calD}-2}{b_{\calD}-1} 
        =\binom{a_{\calD}+b_{\calD}+c_{\calD}-3}{a_{\calD}-1,b_{\calD}-1,c_{\calD}-1},
    \end{align*}
    as
    \[
        \e(\KK[x_1,\dots,x_{a_{\calD}}])=\e(\KK[y_1,\cdots,y_{b_{\calD}}])=\e(\KK[z_1,\cdots,z_{c_{\calD}}])=1. \qedhere
    \]
\end{proof} 

Secondly, we consider the regularity. The following is known.

\begin{Lemma}
    [{\cite[Theorem 5.3]{zbMATH06454819}}]
    \label{MD:5.3}
    Let $S_1,\dots,S_s$ be graded polynomial rings on disjoint sets of variables over $\KK$. For $i=1,\dots,s$, let $M_i$ be a graded finitely generalized Cohen--Macaulay $S_i$-module of positive dimension. 
    \begin{enumerate}[a]
        \item If $\dim(M_i)=1$ for all $i$, then $M_1\uotimes\cdots \uotimes M_s$ is a Cohen--Macaulay $S_1\uotimes \cdots \uotimes S_s$-module, and 
            \[
                \reg(M_1\uotimes \cdots \uotimes M_s)=\max\Set{\reg(M_1),\dots,\reg(M_s)}.
            \]
        \item Assume that at least for one $j$, $\dim(M_j)\ge 2$, and for all $i=1,\dots,s$, $M_i$ is an $\NN$-graded $S_i$-module with $\reg(M_i)<\dim(M_i)$. Then, $M_1\uotimes \cdots \uotimes M_s$ is a Cohen--Macaulay $S_1\uotimes \cdots \uotimes S_s$-module, and
            \begin{align*}
                \reg(M_1\uotimes \cdots \uotimes M_s)& =(\dim(M_1)+\cdots+\dim(M_s)-s+1)\\
                & \qquad -\max\Set{\dim(M_i)-\reg(M_i):1\le i\le s}.  
            \end{align*}
    \end{enumerate}
\end{Lemma} 

\begin{Proposition}
    \label{full-reg}
    Suppose that $\calD$ is the full three-dimensional Ferrers diagram
    $[a_{\calD}]\times [b_{\calD}]\times[c_{\calD}]$. If 
    $a_{\calD}\le b_{\calD}\le c_{\calD}$, then
    \[
        \reg(\calF(I_{\calD}))=a_{\calD}+b_{\calD}-2=\r(I_{\calD}).
    \]
\end{Proposition}

\begin{proof}
    The regularity formula follows from \Cref{MD:5.3} by proceeding as in the proof of \Cref{full-e}. 
    {The reduction number part then follows from \eqref{reg=rednum} in \Cref{Cor:Fib-SR}. 
        Alternatively, \cite[Theorem 5.2]{MR1600012} gives the reduction number.}
\end{proof}

\section{A uniform treatment in the strong projection case}
In this section, we want to provide reasonable estimates of the regularity and the multiplicity of the toric ring associated with the three-dimensional Ferrers diagram. The approach we take here is to consider simultaneously a pair of such diagrams under some conditions. It allows us to investigate these two invariants together.

Recall that for a given lattice point $\bdu\in \calD$, $\calA_{\bdu}(\calD)$ defined in \Cref{def:initial-part} is the subset obtained from $\calD$ by removing the points before $\bdu$ with respect to a given quasi-lexicographic order $\prec$. Meanwhile, we also write $\calA_{\bdu}^{+}(\calD)\coloneqq \calA_{\bdu}(\calD)\setminus \bdu$. 
For simplicity, when the diagram $\calD$ is clear from the context, we will write directly $\calA_{\bdu}$ and $\calA_{\bdu}^{+}$ respectively.

\begin{Definition}
    For the diagram $\calD$ above, we define
    \begin{align*}
        N(\calD)\coloneqq \Set{\bdu\in \calD^{1}: \ini(\bfI_2(\calA_{\bdu}))\supsetneq \ini(\bfI_2(\calA_{\bdu}^{+}))\KK[\bfT_{\calA_{\bdu}}]}
    \end{align*}
    and $\Phan(\calD)\coloneqq  \calD^{1}\setminus N(\calD)$ to be the set of \emph{normal points} and \emph{phantom points} {(}with respect to a chosen quasi-lexicographic order $\prec${)} respectively. Note that the initial ideals are with respect to the lexicographic monomial order on $\KK[\bfT_\calD]$.  
\end{Definition}

These notions were introduced in our previous paper \cite{arXiv:1709.03251}.
Pertinent properties regarding them are summarised in \Cref{Ob3.8Rm3.5} and \Cref{rmk:29}. 

\begin{Proposition}
    [{\cite[Corollary 2.14, Remark 3.5, and Observation 3.8]{arXiv:1709.03251}}]
    \label{Ob3.8Rm3.5}
    Let $\calD$ be a three-dimensional Ferrers diagram that satisfies the projection property.
    For any $\bdu\in \calD^{1}$, we have the following facts. 
    \begin{enumerate}[a]
        \item Let $\KK[\bfT_{\calA_{\bdu}}]$ be the subring of $\KK[\bfT_\calD]$ induced by the containment $\calA_\bdu \subseteq \calD$. Then the $2$-minors ideal $\bfI_2(\calA_{\bdu})$ is precisely $\bfI_2(\calD)\cap \KK[\bfT_{\calA_{\bdu}}]$, and the minimal monomial generating set $\gens(\ini(\bfI_2(\calA_{\bdu})))$ is precisely $\gens(\ini(\bfI_2(\calD)))\cap \KK[\bfT_{\calA_{\bdu}}]$.
            In particular, the restriction complex $\Delta(\calD,\calA_{\bdu})$ is $\Delta(\calA_{\bdu})$.

        \item If $\bdu$ is a phantom point, then trivially
            $\codim\bfI_2(\calA_{\bdu})=\codim\bfI_2(\calA_{\bdu}^{+})$. 
        \item If $\bdu$ is a normal point, then $\dim\Delta(\calA_{\bdu})=\dim\Delta(\calA_{\bdu}^{+})$ and $\codim\bfI_2(\calA_{\bdu})=\codim\bfI_2(\calA_{\bdu}^{+})+1$.
    \end{enumerate}
\end{Proposition}

\begin{Remark}
    \label{rmk:29}
    Let $\calD$ be a three-dimensional Ferrers diagram that satisfies the projection property.  Suppose that a quasi-lexicographic order $\prec$ induces a shedding order on $\Delta(\calD)$ (the induction order $\prec_\calD$ in \Cref{set:IO} satisfies this requirement by \cite[Theorem 4.1]{arXiv:1709.03251}). Now, take an arbitrary $\bdu\in \calD^1$. When $\bdu$ is a phantom point, we observed in \cite[Remark 3.6]{arXiv:1709.03251}
    that $\Delta(\calA_{\bdu})$ is a cone over $\Delta(\calA_{\bdu})\setminus
    T_{\bdu}=\Delta(\calA_{\bdu}^{+})$ with the apex $T_{\bdu}$.
    Trivially we have
    \begin{equation}
        \reg(\Delta(\calA_{\bdu}))= \reg(\Delta(\calA_{\bdu}^{+})) \quad  \text{ and } \quad
        \e(\Delta(\calA_{\bdu}))= \e(\Delta(\calA_{\bdu}^{+})) 
        \label{e-reg-phantom}
    \end{equation}
    in this case. On the other hand, if $\bdu$ is a normal point with respect to $\calD$, then by our assumption, $T_{\bdu}$ is a shedding vertex of $\Delta(\calA_{\bdu})$. Whence, it follows from \eqref{e-seq} and \eqref{reg-seq} in \Cref{rmk:alg} that 
    \begin{align}
        \reg(\Delta(\calA_{\bdu}))&=\max \Set{ \reg(\Delta(\calA_{\bdu}^{+})), \reg(\link_{\Delta(\calA_{\bdu})}(T_{\bdu}))+1}, \label{reg-ind}\\
        \e(\Delta(\calA_{\bdu}))&= \e(\Delta(\calA_{\bdu}^{+}))+ \e(\link_{\Delta(\calA_{\bdu})}(T_{\bdu})). \label{e-ind}
    \end{align}
    In particular, 
    \begin{equation}
        \reg(\Delta(\calA_{\bdu}))\ge \reg(\Delta(\calA_{\bdu}^{+}))
        \quad \text{ and } \quad
        \e(\Delta(\calA_{\bdu}))\ge \e(\Delta(\calA_{\bdu}^{+})). \label{e-reg-compare}
    \end{equation}
\end{Remark}

The synchronizing treatment applied in this section requires the introduction of a stronger condition. The necessity for introducing this condition is discussed in \Cref{rmk:counter-example}.

\begin{Definition}
    \label{def:strongPP}
    Let $\calD$ be a three-dimensional Ferrers diagram. Then $\calD$ is said to satisfy the \emph{strong projection property} if the following equivalent conditions hold for each $i$ in $[a_{\calD}-1]$:
    \begin{enumerate}[a]
        \item for each $\bdu\in\calD^i$, both $\calZ_1^{\ge i+1}(\calD,\bdu)$ and $\calZ_6^{\ge i+1}(\calD,\bdu)$ are empty;
        \item for each $\bdu\in\calD^i$, both $b_{\calD^{i+1}}\le \beta_{\calD}(\bdu)$ and $c_{\calD^{i+1}}\le \gamma_{\calD}(\bdu)$ hold;
        \item $(i,b_{\calD^{i+1}},c_{\calD^i})\in \calD$ and   $(i,b_{\calD^{i}},c_{\calD^{i+1}})\in \calD$.
    \end{enumerate}
\end{Definition}

\begin{Example}
    Consider the diagram $\calD$ in \Cref{MinExam}. Since $(2,b_{\calD^{2}},c_{\calD^{3}})=(2,3,2)\not\in \calD$, this diagram does not satisfy the strong projection property. The minimal
    three-dimensional Ferrers diagram that contains $\calD$ and satisfies the strong projection property, is the diagram $\calD'$, consisting of the following lattice points 
    \begin{align*}
        (1,1,1),
        (1,1,2), 
        (1,1,3), 
        (1,2,1), 
        (1,2,2), 
        (1,2,3), 
        (1,3,1), 
        (1,3,2),
        (1,3,3),\\ 
        (2,1,1), 
        (2,1,2), 
        (2,1,3), 
        (2,2,1),
        (2,2,2),
        (2,2,3),
        (2,3,1), 
        (2,3,2),\\ 
        (3,1,1), 
        (3,1,2), 
        (3,2,1). 
    \end{align*}
\end{Example}

The following facts are clear from the definition. 

\begin{Observation}\label{StrongP}
    \begin{enumerate}[a]
        \item If $\calD$ is a full three-dimensional Ferrers diagram, then it satisfies the strong projection property.
        \item If $\calD$ is a three-dimensional Ferrers diagram and one of $a_{\calD},b_{\calD}$ and $c_{\calD}$ is $1$, then $\calD$ is practically a two-dimensional Ferrers diagram and satisfies the strong projection property.
        \item If $\calD$ is a three-dimensional Ferrers diagram that satisfies the strong projection property, then it satisfies the projection property.
        \item Suppose that $\calD$ is a three-dimensional Ferrers diagram satisfying the strong projection property. Then all the three truncated subdiagrams 
            \[
                \Set{(i,j,k)\in \calD: i\ne i_0}, \quad
                \Set{(i,j,k)\in \calD: j\ne j_0} \text{ and }
                \Set{(i,j,k)\in \calD: k\ne k_0}
            \]
            are \emph{essentially} three-dimensional Ferrers diagrams that still satisfy the strong projection property; see also \Cref{def:essentially}. 
    \end{enumerate}
\end{Observation}

Let $\calD$ be a three-dimensional Ferrers diagram that satisfies the projection property.  In \cite[Theorem 4.1]{arXiv:1709.03251}, we proved that the complex $\Delta(\calD)$ is pure vertex-decomposable of dimension $a_{\calD}+b_{\calD}+c_{\calD}-3$, and the induction order in \Cref{set:IO} gives a shedding order. As a matter of fact, in the early draft of that paper, we showed that the usual lexicographic order gives a shedding order. However, a proof for the primeness of the $2$-minors ideals is hard to achieve if we use the lexicographic order. 

In the following, we will give a direct proof that when $\calD$ satisfies the strong projection property, then the lexicographic order gives a shedding order. Unlike the proof of \cite[Theorem 4.1]{arXiv:1709.03251}, we don't have a second stage to deal with here. Thus, the proof here is relatively shorter than the previous one. In the proof, we will use backward induction on the lattice points of $\calD$ with respect to the lexicographic order. Undoubtedly from the definition of vertex decomposable complexes, we need to check both the link complexes and the deletion complexes are pure of expected dimensions. Counting the related phantom points is our tool for achieving this goal. Of course, we need to know where to find the phantom points. The following observation is summarized from \cite[Discussion 3.7]{arXiv:1709.03251}. 

\begin{Remark}
    \label{dis3.7} 
    Let $\calD$ be a three-dimensional Ferrers diagram satisfying the projection property and assume that $a_{\calD}\ge 2$. The point $\bdu=(1,j_1,k_1)\in\calD^1$ is called a \emph{border point} if $(1,j_1+1,k_1+1)\notin \calD$.
    Let $\calB\subseteq \calD^1$ denote the set of border points of $\calD$. It has the following two special subsets:
    \begin{enumerate}[i]
        \item [\text{$\calB_y:$}] the border points on the
            $y=1,{2},\dots,b_{\calD^{\ge 2}}-1$ lines with minimal
            $z$-coordinates;
        \item [\text{$\calB_z:$}] the border points on the
            $z=1,{2},\dots,c_{\calD^{\ge 2}}-1$ lines with minimal
            $y$-coordinates.
    \end{enumerate}
    Then, $\calB_y$ and $\calB_z$ are disjoint and $\calB\setminus(\calB_y\cup\calB_z)$ is precisely $\Phan(\calD)$. Consequently, the normal points and phantom points are independent of the concrete choice of quasi-lexicographic order.
\end{Remark}

\begin{Example}
    Let $\calD$ be a typical three-dimensional Ferrers diagram that satisfies the projection property.
    In \Cref{phantom2}, the union of the cross-hatch cells 
    \begin{tikzpicture}[scale=0.5]
        \draw[pattern= crosshatch] (0,0)--(0,0.5)--(0.5,0.5)--(0.5,0)--cycle;
    \end{tikzpicture}
    provides $\calB_y$ and the union of the grid cells
    \begin{tikzpicture}[scale=0.5]
        \draw[pattern= grid] (0,0)--(0,0.5)--(0.5,0.5)--(0.5,0)--cycle;
    \end{tikzpicture} provides $\calB_z$. The remaining shaded cells
    \begin{tikzpicture}[scale=0.5]
        \draw[fill=gray!50] (0,0)--(0,0.5)--(0.5,0.5)--(0.5,0)--cycle;
    \end{tikzpicture} give the phantom points. The set $\calB$ of border points is the disjoint union of these three groups. 

    \begin{figure}[htb] 
        \begin{center}
            \scalebox{1.4}{
                \begin{tikzpicture}[scale=0.5, every node/.style={scale=0.6}]
                    \draw (7,0) node [below right]{$y$};
                    \draw (0,6) node [above left]{$z$};                 
                    \draw (-0,3) node [below left] {$c_{\calD^{\geq 2}}-1$};
                    \draw [dotted] (0,3) -- (5,3);
                    \draw (3.5,-0) node [below] {$b_{\calD^{\geq 2}}-1$};
                    \draw [dotted] (3.5,0) -- (3.5,3);
                    \draw [fill=gray!50] (0,5.5)--(0,6)--(0.5,6)--(0.5,5.5)--cycle; 
                    \draw [fill=gray!50] (1,5)--(1,5.5)--(1.5,5.5)--(1.5,5)--cycle; 
                    \draw [fill=gray!50] (1.5,4)--(1.5,5)--(2,5)--(2,4)--cycle; 

                    \draw [fill=gray!50] (5,1.5)--(5,2)--(6.5,2)--(6.5,1.5)--cycle;
                    \draw [fill=gray!50] (6.5,0.5)--(6.5,1)--(7,1)--(7,0.5)--cycle;

                    \draw[pattern= crosshatch] (0,5)--(0,5.5)--(1,5.5)--(1,5)--cycle;
                    \draw[pattern= crosshatch] (1,4.5)--(1,5)--(1.5,5)--(1.5,4.5)--cycle;
                    \draw[pattern= crosshatch] (1.5,3.5)--(1.5,4)--(3.5,4)--(3.5,3.5)--cycle;

                    \draw[fill=gray!50] (3.5,3)--(3.5,4)--(4,4)--(4,3.5)--(5,3.5)--(5,3)--cycle;

                    \draw[pattern= grid] (4.5,1.5)--(4.5,3)--(5,3)--(5,1.5)--cycle; 
                    \draw[pattern= grid] (6,0.5)--(6,1.5)--(6.5,1.5)--(6.5,0.5)--cycle; 
                    \draw[pattern=grid] (6.5,0)--(6.5,0.5)--(7,0.5)--(7,0)--cycle;

                    \draw [very thick] (0,0)--(0,6)--(0.5,6)--(0.5,5.5)--(1.5,5.5)--(1.5,5)--(2,5)--(2,4)--(4,4)--(4,3.5)--(5,3.5)--(5,2)--(6.5,2)--(6.5,1)--(7,1)--(7,0)--cycle; 

                \end{tikzpicture}}
        \end{center}
        \caption{Border points}\label{phantom2}
    \end{figure}
\end{Example}

\begin{Proposition}
    \label{induction-order-strong}
    Let $\calD$ be a three-dimensional Ferrers diagram that satisfies the strong projection property. Then the lexicographic order on $\calD$ gives a shedding order on the pure vertex-decomposable complex $\Delta(\calD)$.
\end{Proposition}

\begin{proof}
    We prove by induction on $a_{\calD}$; this will be called the \emph{outer induction process} in the proof. The base case of the induction is when $a_{\calD}=0$ and $\calD=\varnothing$. The claimed result holds trivially in this case.

    For the induction step of the outer induction process, in the following, we will assume that $a_{\calD}\ge 1$. It suffices to prove that $\Delta(\calA_\bdu(\calD))$ is pure vertex-decomposable, for each $\bdu=(1,j_0,k_0)\in \calD^1$. As a reminder, in this proof, both $\calA$ and $\calA^+$ are with respect to the lexicographic order. 
    We will prove this by backward induction with respect to the lexicographic order; this will be called the \emph{inner induction process} in the proof. The base case of the inner induction process is when we remove the whole $x=1$ layer $\calD^1$ and get $\calD^{\ge 2}$; whence, $\bdu$ is indeed $(2,1,1)$. As $a_{\calD^{\ge 2}}=a_{\calD}-1$ and $\calD^{\ge 2}$ essentially still satisfies the strong projection property, by induction on $a_\calD$, $\Delta(\calD^{\ge 2})$ is pure vertex-decomposable. This establishes the validity in the base case for the inner induction process.

    For the induction step of the inner induction process, in the following, we take a general point $\bdu=(1,j_0,k_0)\in \calD^1$.
    Without loss of generality, we may assume that $\bdu$ is a normal point with respect to $\calD$. As the lexicographic order is automatically a quasi-lexicographic order, the restriction complex  $\Delta(\calD,\calA_{\bdu}(\calD))$ is $\Delta(\calA_{\bdu}(\calD))$ by \Cref{Ob3.8Rm3.5} (a). Similarly, we have $\Delta(\calD,\calA_{\bdu}^{+}(\calD))=\Delta(\calA_{\bdu}^{+}(\calD))$.

    Firstly, we deal with the deletion complex $\Delta(\calA_{\bdu}(\calD))\setminus T_{\bdu}= \Delta(\calA_{\bdu}^{+}(\calD))$.  
    Note that for any $\bdv\in \calD^1\cap \calA_{\bdu}^+(\calD)$, the restriction complex $\Delta(\calD,\calA_{\bdv}(\calD))=\Delta(\calA_{\bdv}(\calD))$ is pure vertex decomposable with respect to the lexicographic order, by induction and by \Cref{Ob3.8Rm3.5} (a). Thus, by applying \Cref{Ob3.8Rm3.5} (b) and (c) repeatedly, we also have
    \begin{equation*}
        \dim(\Delta(\calA_{\bdu}^{+}(\calD)))=\dim(\calD^{\ge 2})+
        \#(\Phan(\calD)\cap \calA_{\bdu}^+(\calD));
    \end{equation*}
    note that only phantom points provide dimensional change.
    Here, we use $\#$ to denote the cardinality of the corresponding set.
    Notice that the phantom points do not depend on the concrete choice of quasi-lexicographic order by \Cref{dis3.7}. It follows from both \Cref{Ob3.8Rm3.5} and \cite[Theorem 4.1]{arXiv:1709.03251} that
    \[
        \dim(\Delta(\calD))=\dim(\Delta(\calD^{\ge 2}))+\#(\Phan(\calD)).
    \]
    Consequently,
    \begin{equation}
        \dim(\Delta(\calA_{\bdu}^{+}(\calD)))=\dim(\Delta(\calD))-\#(\Phan(\calD)\setminus \calA_{\bdu}^+(\calD)).
        \label{eqn:dim-Au-plus}
    \end{equation}
    Since $\bdu$ is not a phantom point, we actually have
    \begin{equation}
        \dim(\Delta(\calA_{\bdu}^{+}(\calD)))=\dim(\Delta(\calD))-\#(\Phan(\calD)\setminus \calA_{\bdu}(\calD)).
        \label{eqn:dim-Au-plus-2}
    \end{equation}

    Next, we consider the link complex $\calL_{\bdu}(\calD):=\link_{\Delta(\calA_{\bdu}(\calD))}(T_{\bdu})$. It suffices to show that $\calL_{\bdu}(\calD)$ is pure vertex decomposable of dimension $\dim(\Delta(\calA_{\bdu}^{+}(\calD)))-1$, and the lexicographic order gives a shedding order.
    Notice that $\calD$ satisfies the strong projection property. Now, we define
    \begin{equation}
        \calH\coloneqq \calZ_3^{\ge 2}(\calD,\bdu)\cup \calZ_5^{1}(\calD,\bdu)\cup \calZ_6^1(\calD,\bdu).
        \label{eqn:lex-H}
    \end{equation}
    Then, $\calL_{\bdu}(\calD)$ is the join of the restriction complex $\Delta(\calD,\calH)$ with a simplex of dimension $\gamma-k_0-1$ for $\gamma=\gamma_{\calD}(\bdu)$; cf.~the detailed calculation in the first stage proof of \cite[Theorem 4.1]{arXiv:1709.03251}. The simplex here corresponds to the set $\Set{(1,j_0,k):k_0<k\le \gamma}$.  Since this $\calH$ satisfies the detaching condition in \cite[Proposition 2.17]{arXiv:1709.03251}, we can use it to deduce that $\Delta(\calH)$ agrees with the restriction complex $\Delta(\calD,\calH)$.

    To study $\Delta(\calH)$, we turn to consider $\calD':=\calZ_3(\calD,\bdu)\cup \calZ_5^1(\calD,\bdu)\cup \calZ_6^1(\calD,\bdu)$. Notice that $\calD'$ is a three-dimensional Ferrers diagram that still satisfies the strong projection property by \Cref{StrongP} (d). The restriction of the lexicographic order on $\calD$ to $\calD'$ is surely the lexicographic order on $\calD'$. Furthermore, $\calH=\calA_{\bdu}^{+}(\calD')$. Since $\calH$ has fewer cells than $\calA_\bdu(\calD)$, by induction, $\Delta(\calH)$ is pure vertex-decomposable and the lexicographic order gives a shedding order. Now, it remains to show that
    \begin{equation}
        \dim(\Delta(\calH))+(\gamma-k_0)=\dim(\Delta(\calA_{\bdu}^+(\calD)))-1,
        \label{eqn:dim-H}
    \end{equation}
    since the left-hand side gives the dimension of $\calL_\bdu(\calD)$.
    Similar to \eqref{eqn:dim-Au-plus}, we have
    \begin{equation}
        \dim(\Delta(\calH))=\dim(\Delta(\calD'))-\#(\Phan(\calD')\setminus \calA_{\bdu}^{+}(\calD')).
        \label{eqn:dim-H-2}
    \end{equation}
    As $a_{\calD'}=a_{\calD}$, $ b_{\calD'}=b_{\calD}$ and $ c_{\calD'}=k_0$, we have
    \[
        \dim(\Delta(\calD))-\dim(\Delta(\calD'))=c_{\calD}-k_0.
    \]
    By \Cref{dis3.7}, we have

    \[
        \Phan(\calD')\setminus\calA_{\bdu}^{+}(\calD')=\Set{(1,j,k_0): b_{(\calD')^{\ge 2}}\le j\le j_0}.
    \]
    However, $b_{(\calD')^{\ge 2}}=\min(j_0,b_{\calD^{\ge 2}})$. Thus,
    \begin{equation}
        \#(\Phan(\calD')\setminus\calA_{\bdu}^{+}(\calD'))= j_0-\min(j_0,b_{\calD^{\ge 2}})+1.
        \label{eqn:phan-D-prime}
    \end{equation}
    By combining equations \eqref{eqn:dim-Au-plus-2}, \eqref{eqn:dim-H}, \eqref{eqn:dim-H-2}, and \eqref{eqn:phan-D-prime} together, it is clear that we have to show 
    \begin{equation}
        \#(\Phan(\calD)\setminus \calA_{\bdu}(\calD))=c_{\calD}+j_0-\min(j_0,b_{\calD^{\ge 2}})-\gamma.
        \label{final-eqn}
    \end{equation}
    We will write 
    $\calQ\coloneqq \Set{(1,j,k)\in \calB:j<j_0}$.
    Obviously, to show \eqref{final-eqn}, we have two cases.
    \begin{enumerate}[a]
        \item Suppose that $b_{\calD^{\ge 2}}\ge j_0$.  As $\calD$ satisfies the strong projection property, we have $\gamma=c_{\calD}$. Therefore, we deduce immediately from \Cref{dis3.7} that 
            \begin{align*}
                \Phan(\calD)\setminus\calA_{\bdu}(\calD)\subseteq \calQ\setminus \calB_y
                =\Set{(1,j,\gamma)\in \calB:j<j_0}\setminus \calB_y=\varnothing,
            \end{align*}
            which gives the desired formula \eqref{final-eqn}.

        \item Suppose instead that $b_{\calD^{\ge 2}}<j_0$.
            If we picture the set of border points $\calB$ as in \Cref{phantom2}, then the southeast extremal point of $\calQ$ is $(1,j_0-1,\gamma)$. Thus, it is not difficult to check that $\#\calQ=c_{\calD}-\gamma+j_0-1$. Notice that $\calB_y\subset \calQ$ with $\# \calB_y=b_{\calD^{\ge 2}}-1$ by \Cref{dis3.7}. Now, as $\calD$ satisfies the strong projection property, $\gamma\ge c_{\calD^{\ge 2}}$.  Since for any $(1,j,k)\in \calQ$,  one has $k\ge \gamma$. Thus, $\calB_z\cap \calQ=\varnothing$. Consequently, $\calQ\setminus\calB_y=\Phan(\calD)\setminus \calA_{\bdu}(\calD)$, giving the desired formula \eqref{final-eqn}. This completes our proof of \Cref{induction-order-strong}.  \qedhere
    \end{enumerate}
\end{proof}

We are now ready to estimate the regularity and multiplicity of three-dimensional Ferrers diagrams that satisfy the strong projection property. In particular, we find an upper bound of those invariants with the help of \Cref{full-e} and \Cref{full-reg}; see \Cref{twoD}. 

To provide such a reasonable estimate, we 
consider simultaneously two diagrams of this type. We argue by induction with respect to the lexicographic order and consider the estimation problems for the accompanied subdiagrams of these two. For that purpose, we need to analyze and compare different zones defined in \Cref{def:zones} of the given subdiagrams. 

\begin{Theorem}
    \label{two-strong}
    Let $\calD_1\subseteq \calD_2$ be two three-dimensional Ferrers diagrams that satisfy the {strong} projection property.  Then, we have
    \begin{equation}
        \r(I_{\calD_1})=\reg(\Delta(\calD_1)) \le \r(I_{\calD_2})=\reg(\Delta(\calD_2))
        \qquad \text{and}\qquad 
        \e(\Delta(\calD_1)) \le \e(\Delta(\calD_2)). \label{DD}
    \end{equation}
\end{Theorem}

\begin{proof}
    Once we have the regularity relationship, the reduction-number part follows from \Cref{Cor:Fib-SR}. Thus, in the following, we will focus on establishing \eqref{DD} for regularity and multiplicity. 
    We prove this by induction on $a_\calD$; this is the \emph{outer induction process} of our proof. Its base case is when $a_{\calD}=0$ and $\calD=\varnothing$. Whence, these inequalities hold trivially. In the following, we assume that $a_{\calD}\ge 1$.

    We will consider simultaneously the lexicographic order on both $\calD_1$ and $\calD_2$.  For each $\bdu\in \calD_1^1$, we will prove by backward induction with respect to the lexicographic order that
    \begin{equation}
        \reg(\Delta(\calA_{\bdu}(\calD_1))) \le \reg(\Delta(\calA_{\bdu}(\calD_2))) 
        \quad \text{and} \quad
        \e(\Delta(\calA_{\bdu}(\calD_1))) \le \e(\Delta(\calA_{\bdu}(\calD_2))). \label{reg-e-ineq-u}
    \end{equation}
    This will be the \emph{inner induction process} of our proof.
    Note that when $\bdu=(1,1,1)$, we obtain the expected inequalities in \eqref{DD}.
    As a reminder, in this proof, both $\calA_\bdu$ and $\calA_\bdu^+$ are with respect to the lexicographic order. Now, we carry out the inner induction argument.
    \begin{flushleft}
        \textbf{Base case}
        \par\end{flushleft}
    The base case of this inner induction process is when we remove the whole $x=1$ layer so that indeed $\bdu=(2,1,1)\in \calD_1$; without loss of generality, we assume that $\calD^{\ge 2}\ne \varnothing$. By induction on $a_{\calD_1}$, we surely will have
    \[
        \reg(\Delta(\calD_1^{\ge 2}))\le \reg(\Delta(\calD_2^{\ge 2}))
        \quad \text{and}\quad
        \e(\Delta(\calD_1^{\ge 2}))\le \e(\Delta(\calD_2^{\ge 2})),
    \]
    establishing the validity in the base case.
    Note that when $\calD^{\ge 2}=\varnothing$, these inequalities hold trivially.
    \begin{flushleft}
        \textbf{Induction step}
        \par\end{flushleft}  
    Now, consider a general point $\bdu\in \calD_1^1$. Let $\bdv\in \calD_1$ such that $\calA_{\bdu}^{+}(\calD_1)=\calA_{\bdv}(\calD_1)$.  Obviously, $\bdu$ precedes $\bdv$ lexicographically. Since \Cref{induction-order-strong} confirms that the lexicographic order gives a shedding order to both $\Delta(\calD_1)$ and $\Delta(\calD_2)$, we can apply the calculations in \Cref{rmk:29}.
    \begin{enumerate}[a]
        \item Suppose that $\bdu$ is a phantom point of $\calD_1$. Now, by induction, equalities \eqref{e-reg-phantom} and \eqref{e-reg-compare}, we have
            \begin{align*} 
                \reg(\calA_{\bdu}(\calD_1))&=\reg(\calA_{\bdu}^{+}(\calD_1))
                =\reg(\calA_{\bdv}(\calD_1))
                \le \reg(\calA_{\bdv}(\calD_2)) 
                \le \reg(\calA_{\bdu}(\calD_2)), 
            \end{align*}
            establishing \eqref{reg-e-ineq-u} for the regularity in this case.  Similarly, we can deduce the inequality for multiplicity.        

        \item If $\bdu$ is a normal point with respect to $\calD_1$, then it is also a normal point with respect to ${\calD_2}$ by definition. Hence, by induction and \eqref{e-reg-compare}, we have
            \begin{align*}
                \reg(\Delta(\calA_{\bdu}^{+}(\calD_1)))=  \reg(\Delta(\calA_{\bdv}(\calD_1))) 
                \le \reg(\Delta(\calA_{\bdv}(\calD_2))\le \reg(\Delta(\calA_{\bdu}^{+}(\calD_2))).
            \end{align*}
            Similarly, we have 
            \[
                \e(\Delta(\calA_{\bdu}^{+}(\calD_1))) \le \e(\Delta(\calA_{\bdu}^{+}(\calD_2))).
            \]
            For $s=1,2$, write $\calL_{\bdu}(\calD_s):=\link_{\Delta(\calA_{\bdu}(\calD_s))}(\bfT_{\bdu})$.  As in the proof of \Cref{induction-order-strong}, we know that each $\calL_{\bdu}(\calD_s)$ is a join of a simplex with $\Delta(\calH(\calD_s))$, where 
            \begin{equation} 
                \calH(\calD_s):=\calZ_3^{\ge 2}(\calD_s,\bdu)\cup \calZ_5^1(\calD_s,\bdu)\cup \calZ_6^1(\calD_s,\bdu). 
                \label{eqn-def-H} 
            \end{equation} 
            If we write 
            \begin{equation*}
                \calD_s':=\calZ_3(\calD_s,\bdu)\cup \calZ_5^1(\calD_s,\bdu)\cup \calZ_6^1(\calD_s,\bdu),
            \end{equation*}
            then $\calH(\calD_s)=\calA_{\bdu}^{+}(\calD_s')$.  Note that for each $s$, the set $\calD_s'$ is a three-dimensional Ferrers diagram that satisfies the strong projection property {by \Cref{StrongP} (d).}
            And obviously, we have the corresponding containment $\calD_1'\subseteq \calD_2'$.
            Now, for each $s$, we have $\calH(\calD_s)=\calA_{\bdu}^{+}(\calD_s')=\calA_{\bdw_s}(\calD_s')$ for some $\bdw_s\in \calD_s'$.
            As $\calH(\calD_1)=\calH(\calD_2)\cap \calD_1$, it is clear that $\bdw_2$ precedes $\bdw_1$ with respect to the lexicographic order. Thus, by induction and \eqref{e-reg-compare}, we have
            \begin{align*}
                \reg(\calL_{\bdu}(\calD_1))&=\reg(\Delta(\calH(\calD_1)))=\reg(\Delta(\calA_{\bdw_1}(\calD_1'))) \le \reg(\Delta(\calA_{\bdw_1}(\calD_2')))\\
                &\le \reg(\Delta(\calA_{\bdw_2}(\calD_2')))
                =\reg(\Delta(\calH(\calD_2))) =\reg(\calL_{\bdu}(\calD_2)).
            \end{align*}
            And similarly, we have
            \begin{align*}
                \e(\calL_{\bdu}(\calD_1))\le \e(\calL_{\bdu}(\calD_2)).
            \end{align*}
            However, by \eqref{reg-ind} and \eqref{e-ind}, we have
            \begin{align*}     
                \reg(\Delta(\calA_{\bdu}(\calD_s)))&=\max\Set{\reg(\Delta(\calA_{\bdu}^{+}(\calD_s))),\reg(\calL_{\bdu}(\calD_s))+1 },\\
                \e(\Delta(\calA_{\bdu}(\calD_s)))&=\e(\Delta(\calA_{\bdu}^{+}(\calD_s)))+\e(\calL_{\bdu}(\calD_s)), 
            \end{align*}
            for $s=1,2$. Putting the above data together, we can get the desired inequalities in \eqref{reg-e-ineq-u}.
            This completes our proof of \Cref{two-strong}. 
            \qedhere
    \end{enumerate}
\end{proof}

\begin{Remark}
    \label{rmk:counter-example}
    In the proof above, the \emph{strong projection property} condition, instead of the \emph{projection property} condition, is used to get the equality $\calH(\calD_1)=\calH(\calD_2)\cap \calD_1$. Otherwise, the subsets $\calH(\calD_s)$ for the link complexes won't be as simple as in \eqref{eqn-def-H}. Instead, one has to use the formula \eqref{eqn:complicated-H}  in the next section. Whence, it might be possible that $\calH(\calD_1)\nsubseteq \calH(\calD_2)$ and $\e(\calL_{\bdu}(\calD_1))>\e(\calL_{\bdu}(\calD_2))$.
    For instance, one can take $\calD_1$ to be the minimal three-dimensional Ferrers diagram containing $(1,3,2)$ and $(2,2,3)$, and take $\calD_2$ to be the full diagram $[2]\times[3]\times [3]$. 
   The diagram $\calD_1$ satisfies the projection property. But the strong projection property is not satisfied, since $(1,b_{\calD_1^1},c_{\calD_1^2})=(1,3,3)\notin \calD_1$. For $\bdu=(1,3,1)\in \calD_1$, we will have
    \[
        \e(\calL_{\bdu}(\calD_1))=2>\e(\calL_{\bdu}(\calD_2))=1.
    \]
\end{Remark} 

In the following, we consider an easy application of \Cref{two-strong}.

\begin{Definition}
    Let $\calD$ be a three-dimensional Ferrers diagram. Its \emph{$(x,y)$-profile} is the two-dimensional Ferrers diagram $\calP_{x,y}(\calD):=\Set{(i,j):(i,j,k)\in \calD}$. In a similar vein, one can define the $(x,z)$-profile $\calP_{x,z}(\calD)$.
\end{Definition}  

\begin{Example} \label{twoD}
    Let $\calD$ be a three-dimensional Ferrers diagram that satisfies the strong projection property.
    \begin{enumerate}[a]
        \item Let $\overline{\calP}_{x,y}:=\calP_{x,y}(\calD)\times [c_{\calD}]$. Then the pair $\calD\subseteq \overline{\calP}_{x,y}$ satisfies the assumptions in \Cref{two-strong}. Thus, by \Cref{MD:5.3} and \Cref{lem:14}, we have
            \begin{align*}
                \reg(\calF(I_{\calD}))&\le \reg(\calF(I_{\overline{\calP}_{x,y}}))\\
                &=a_{\calD}+b_{\calD}+c_{\calD}-2-\max\{a_{\calD}+b_{\calD}-1-\reg(\calF(I_{\calP_{x,y}(\calD)})),c_{\calD}\},  \\
                \intertext{and}
                \e(\calF(I_{\calD}))&\le \e(\calF(I_{\overline{\calP}_{x,y}}))
                =\binom{a_{\calD}+b_{\calD}+c_{\calD}-3}{c_{\calD}-1}\e(\calF(I_{\calP_{x,y}(\calD)})).
            \end{align*}
            Notice that $\calP_{x,y}(\calD)$ is a two-dimensional Ferrers diagram.  The associated regularity and multiplicity have been computed in Proposition 5.7 and Corollary 5.6 of \cite{MR2457403} respectively.

        \item The pair $\calD\subseteq \calX:=[a_{\calD}]\times[b_{\calD}]\times[c_{\calD}]$ satisfies the assumptions in \Cref{two-strong}. Thus, we have
            \begin{align}
                \reg(\calF(I_{\calD}))&\le \reg(\calF(I_{\calX}))=\min(a_{\calD}+b_{\calD},a_{\calD}+c_{\calD},b_{\calD}+c_{\calD})-2, \label{reg-D-X} \\
                \intertext{and}
                \e(\calF(I_{\calD}))&\le \e(\calF(I_{\calX}))=\binom{a_{\calD}+b_{\calD}+c_{\calD}-3}{a_{\calD}-1,b_{\calD}-1,c_{\calD}-1},  \notag\label{e-D-X}
            \end{align}
            by the multiplicity and regularity calculated in \Cref{full-e} and \Cref{full-reg} respectively.
    \end{enumerate}
\end{Example}

\section{General case} 
The main purpose of this section is to show that the inequality \eqref{reg-D-X} can be generalized to the case when $\calD$ only satisfies the projection property. Since we drop the stronger assumption in \Cref{two-strong}, the proof here will be inevitably more involved. 

We will use the induction order outlined in \Cref{set:IO}. It gives us a shedding order of the vertex decomposable complex $\Delta(\calD)$ by \cite[Theorem 4.1]{arXiv:1709.03251}. We will apply the standard techniques summarized in \Cref{rmk:29}. As we don't assume the strong projection property here, the induction order requires a two-stage treatment.

Throughout this section, for a three-dimensional Ferrers diagram $\calD$, we write $\mu_{\calD}:=\min(a_{\calD}+b_{\calD},a_{\calD}+c_{\calD},b_{\calD}+c_{\calD})$.

\begin{Theorem}
    \label{thm-reg}
    Let $\calD$ be a three-dimensional Ferrers diagram that satisfies the projection property. Then, 
    \begin{equation}
        \reg(\Delta(\calD))\le \mu_\calD-2.
        \label{reg-bound}
    \end{equation}
    In particular, $\r(I_{\calD})=\reg(\calF(I_\calD))\le \mu_{\calD}-2$.
\end{Theorem}

\begin{proof}
    The ``in particular'' part follows from \Cref{Cor:Fib-SR} and the inequality \eqref{reg-bound}. Thus, in the following, we will focus on proving \eqref{reg-bound}. For this purpose, we prove by induction on $a_\calD$; this will be the \emph{outer induction process} in the proof.

    The base case of the outer induction process is when $a_{\calD}=1$. In this situation, the diagram $\calD$ trivially satisfies the strong projection property. Thus, the inequality \eqref{reg-bound} follows from  \eqref{reg-D-X} and \Cref{full-reg}. 
    Therefore, in the following, we will assume that $a_{\calD}\ge 2$. By symmetry, we will also assume that $b_\calD$ and $c_{\calD}$ are at least $2$.

    To achieve the inequality \eqref{reg-bound}, we use backward induction with respect to the order outlined in \Cref{set:IO} and prove that
    \begin{equation}
        \reg(\Delta(\calA_{\bdu}(\calD)))\le \mu_{\calD}-2
        \label{reg-ineq}
    \end{equation} 
    for each $\bdu\in \calD^1$ in the \emph{first stage} with respect to $\calD$.
    Then, the claimed estimate \eqref{reg-bound} follows when we choose $\bdu=(1,1,1)$. This will be the \emph{inner induction process} of this proof.

    As a reminder, within this proof, both $\calA$ and $\calA^+$ are with respect to this induction order, unless explicitly stated otherwise. Notice that this total order gives us a shedding order of the vertex decomposable complex $\Delta(\calD)$ by \cite[Theorem 4.1]{arXiv:1709.03251}. Hence, we can apply the standard techniques summarized in \Cref{rmk:29}. Furthermore, still from the proof of \cite[Theorem 4.1]{arXiv:1709.03251}, we can see that $\Delta(\calA_{\bdv}(\calD))$ is pure vertex-decomposable and coincides with the restriction of $\Delta(\calD)$ to $\calA_{\bdv}(\calD)$ for each $v\in \calD$. Now, we start the induction argument of the inner induction process. 
    \begin{flushleft}
        \textbf{Base case}
        \par\end{flushleft}
    We have two subcases here.
    \begin{enumerate}[a]
        \item Suppose that the points of the second stage exist, i.e., $\calC_2({\calD})\ne \varnothing$. Then, we need to prove \eqref{reg-ineq} for $\bdu=(1,1,c_{{\calD}^{\ge 2}}+1)$, the initial element in the second stage. For that purpose, we flip ${\calD}$ to get $\calS({\calD})$ and write it as ${\calD}'$ for simplicity; see also \Cref{set:IO}.  Then, $\calS(\calA_\bdv({\calD}))$ is precisely $\calA_{\calS(\bdv)}({\calD}')$ for each $\bdv\in \calC_2(\calD)$ in the first stage; the $\calA$ in $\calA_{\calS(\bdv)}({\calD}')$ is with respect to the \emph{lexicographic order} on $\calD'$. 
            Consequently, $\Delta(\calD',\calA_{\calS(v)}(\calD'))=\Delta(\calA_{\calS(\bdv)}(\calD'))$ is vertex-decomposable and the lexicographic order gives a shedding order (at least for the points in the $x=1$ layer).
            Since $b_{({\calD}')^{\ge 2}}<b_{{\calD}'}$ while $\mu_{{\calD}}=\mu_{{\calD}'}$, we can apply subsequent \Cref{lex-layer-1} to complete the proof.   

        \item Suppose instead that $\calC_2({\calD}) = \varnothing$. Then, we need to prove \eqref{reg-ineq} for $\bdu=(2,1,1)$, the initial element of the second layer; without loss of generality, we assume that $\calD^{\ge 2}\ne \varnothing$. By induction on $a_{\calD}$ in the outer induction process, we surely will have
            \[
                \reg(\Delta(\calD^{\ge 2}))\le \mu_{\calD^{\ge 2}}-2
            \]
            from \eqref{reg-bound}. Notice that $a_{\calD^{\ge 2}}=a_{\calD}-1$ while $b_{\calD^{\ge 2}}\le b_{\calD}$ and $c_{\calD^{\ge 2}}\le c_{\calD}$. Consequently, $\mu_{\calD^{\ge 2}}\le \mu_{\calD}$, establishing the validity of \eqref{reg-ineq} in the base case. 
    \end{enumerate}   
    \begin{flushleft}
        \textbf{Induction step}
        \par\end{flushleft}
    Now, take a general $\bdu=(1,j_0,k_0)\in \calD^1$ in the first stage. Without loss of generality, we assume that $\bdu$ is a normal point with respect to $\calD$. By induction, we have $\reg(\Delta(\calA_{\bdu}^{+}(\calD)))\le \mu_{\calD}-2$. Thus, in view of equation \eqref{reg-ind} in \Cref{rmk:29}, it suffices to show that 
    \begin{equation*}
        \reg(\link_{\Delta(\calA_{\bdu}(\calD))}(T_{\bdu}))\le \mu_{\calD}-3.
    \end{equation*} 
    As explicitly shown in the proof of \cite[Theorem 4.1]{arXiv:1709.03251},  the link complex $\link_{\Delta(\calA_{\bdu}(\calD))}(T_{\bdu})$ is the join of a simplex with $\Delta(\calH)$, where the subset 
    \begin{equation}
        \calH\coloneqq \calZ_1^{\ge 2}(\calD,\bdu) \cup \calZ_3^{\ge 2}(\calD,\bdu) \cup \calZ_5^1(\calD,\bdu)\cup \calZ_6^1 (\calD,\bdu) 
        \cup \Set{(1,j,k)\in \calD: j\le j_0 \text{ and } k>c_{\calD^{\ge 2}}};
        \label{eqn:complicated-H}
    \end{equation}
    see also \Cref{exam:6.2}.
    As in the proof of \cite[Lemma 4.3]{arXiv:1709.03251}, we turn to consider the diagram
    \begin{equation}
        \widetilde{\calD}\coloneqq \calZ_3(\calD,\bdu) \cup \calZ_5^1(\calD,\bdu) \cup \calZ_6^{1}(\calD,\bdu) \cup \Set{(i,j,k)\in \calD: j\le j_0 \text{ and } k>\min(\gamma,c_{\calD^{\ge 2}})} 
        \label{eqn:tilde-D}
    \end{equation}
    for $\gamma=\gamma_{\calD}(\bdu)$.
    The new concise diagram $\widetilde{\calD}$ is \emph{essentially} a three-dimensional Ferrers diagram satisfying the projection property. In addition, $\widetilde{\calD}$ contains $\calH$ and $\calH=\calA_{\bdu}^{+}(\widetilde{\calD})$. The inequality that we want to show is then translated into
    \begin{equation}
        \reg(\Delta(\calH))\le \mu_{\calD}-3.
        \label{reg-ineq-colon-H}
    \end{equation} 
    \begin{enumerate}[a]
        \item Suppose that $\calH^1=\varnothing$. Since we have assumed earlier that $b_\calD\ge 2$, this happens precisely when $j_0=b_{\calD}>1$ and $c_{\calD^{\ge 2}}=c_{\calD}$.
            Whence,
            $\calH=\widetilde{\calD}^{\ge 2}=\calZ_1^{\ge 2}(\calD,\bdu)\cup \calZ_3^{\ge 2}(\calD,\bdu)$. Note that we always have $k_0\le \gamma$.
            \begin{enumerate}[i]
                \item If $k_0=\gamma$, then $\bdu$ is a phantom point.  But we have assumed that $\bdu$ is a normal point. This cannot happen.
                \item If $k_0<\gamma$, then $a_{\widetilde{\calD}^{\ge 2}}= a_{\calD}-1$, $b_{\widetilde{\calD}^{\ge 2}}\le b_{\calD}$ and $c_{\widetilde{\calD}^{\ge 2}}\le c_{\calD}-(\gamma-k_0)<c_{\calD}$. 
                    Since $\calH$ has fewer points than $\calA_\bdu(\calD)$, by induction, 
                    \begin{align*}
                        \reg(\link_{\Delta(\calA_{\bdu}(\calD))}(T_{\bdu}))& =\reg(\Delta(\calH))=\reg(\Delta(\calH^{\ge 2}))
                        \le \mu_{\widetilde{\calD}^{\ge 2}}-2<\mu_{\calD}-2, 
                    \end{align*}
                    establishing the expected inequality \eqref{reg-ineq-colon-H} in this case.
            \end{enumerate}

        \item Suppose instead that $\calH^1\ne \varnothing$.  We will prove in \Cref{reg-lem} that $\reg(\Delta(\calH))\le\mu_{\widetilde{\calD}}-3$. As obviously $\mu_{\widetilde{\calD}}\le\mu_\calD$, we still get the expected inequality \eqref{reg-ineq-colon-H} in this case. And this completes our proof of \Cref{thm-reg}. \qedhere
    \end{enumerate}
\end{proof}

Before fully proving \Cref{thm-reg} by establishing the accompanied lemmas, we provide an example with all regions and sub-diagrams involved in the proof of the theorem.

\begin{Example}
    \label{exam:6.2}
    Let $\calD$ be a typical three-dimensional Ferrers diagram that satisfies the projection property. Let $\bdu$ be a point in the first stage. When
    the subset $\calH$ introduced in \eqref{eqn:complicated-H} is non-empty, 
    in \Cref{H}, we present the $\calH^1$ by the shaded regions with two different cases: $\gamma\leq c_{\calD^{\ge 2}}$ (left-hand side) and $\gamma>c_{\calD^{\ge 2}}$ (right-hand side). The six zones are with respect to $\calD$ and $\bdu$.  Meanwhile, \Cref{D} shows the corresponding regions of $\widetilde{\calD}^1$ for the diagram $\widetilde{\calD}$ introduced in \eqref{eqn:tilde-D}. Note that $\calH^{\geq 2}=\widetilde{\calD}^{\geq 2}$, which is not pictured here.
    \begin{figure}[htb] 
        \begin{center}
            \scalebox{1.4}{
                \begin{tikzpicture}[thick, scale=0.5, every node/.style={scale=0.6}]

                    \draw (7,0) node [below right]{$y$};
                    \draw (0,6) node [above left]{$z$};  

                    \path (0,0)--(0,4)  node [below left]{$\gamma$};

                    \path (0,-0.2)--(5,-0.2) node [below left]{$\beta$};

                    \draw (0,4)--(2,4)--(2,5)--(1.5,5)--(1.5,5.5)--(0.5,5.5)--(0.5,6)--(0,6)--cycle;
                    \draw [fill=gray!50] (0,5)--(1.5,5)--(1.5,5.5)--(0.5,5.5)--(0.5,6)--(0,6)--cycle;
                    \node [right] at (0.5,4.7) {$\mathcal{Z}_1$};
                    \draw [dotted] (0,5) -- (1.5,5);
                    \path (0,4.7)--(0,4.7) node [ left]{$c_{\calD^{\geq 2}}$};
                    \draw  (0,2.5)--(2.5,2.5)--(2.5,4)--(0,4)--cycle;
                    \node [right] at (0.5,3.3) {$\mathcal{Z}_2$};
                    \draw (10,3)--(10,2.5);

                    \draw  (0,0)--(2.5,0)--(2.5,2.5)--(0,2.5)--cycle;

                    \node [right] at (0.5,1.2) {$\mathcal{Z}_3$};

                    \draw [pattern=crosshatch] (2.5,2.5)--(2,2.5)--(2,2)--(2.5,2)--cycle; 
                    \node [above right] at (4,5) {$(j_0,k_0)$}; 
                    \draw [->, thin] (2.25,2.25) to [bend left] (4,5);
                    \draw [dotted] (2,0) -- (2,5);
                    \path (1.7,0)--(1.7,0) node [below]{$b_{\calD^{\geq 2}}$};
                    \draw  (2.5,2.5)--(5,2.5)--(5,3.5)--(4,3.5)--(4,4)--(2.5,4)--cycle;
                    \node [right] at (3.2,3.3) {$\mathcal{Z}_4$};

                    \draw (2.5,0)--(5,0)--(5,2.5)--(2.5,2.5)--cycle;

                    \draw [fill=gray!50] (2.5,0)--(5,0)--(5,2.5)--(2.5,2.5)--cycle;
                    \node [right] at (3.2,1.2) {$\mathcal{Z}_5$};
                    \draw   (5,0)--(7,0)--(7,1)--(6.5,1)--(6.5,2)--(5,2)--cycle;

                    \draw [fill=gray!50] (5,0)--(7,0)--(7,1)--(6.5,1)--(6.5,2)--(5,2)--cycle;
                    \draw (5.9,1.2) node {$\mathcal{Z}_6$};

                    \draw (17,0) node [below right]{$y$};
                    \draw (10,6) node [above left]{$z$};  

                    \path (10,0)--(10,4) node [below left]{$\gamma$};

                    \path (10,-0.2)--(15,-0.2) node [below left]{$\beta$};

                    \draw [fill=gray!50] (10,4)--(10,6)--(10.5,6)--(10.5,5.5)--(11.5,5.5)--(11.5,5)--(12,5)--(12,4)--cycle;
                    \node [right] at (10.5,4.7) {$\mathcal{Z}_1$};
                    \draw [dotted] (10,3) -- (11.5,3);
                    \path (10,2.7)--(10,2.7) node [ left]{$c_{\calD^{\geq 2}}$};

                    \draw [fill=gray!50] (10,4)--(12.5,4)--(12.5,3)--(10,3)--cycle;
                    \node [right] at (10.5,3.3) {$\mathcal{Z}_2$};

                    \draw  (10,0)--(12.5,0)--(12.5,2.5)--(10,2.5)--cycle;

                    \node [right] at (10.5,1.2) {$\mathcal{Z}_3$};

                    \draw [pattern=crosshatch] (12.5,2.5)--(12,2.5)--(12,2)--(12.5,2)--cycle; 
                    \node [above right] at (14,5) {$(j_0,k_0)$}; 
                    \draw [->, thin] (12.25,2.25) to [bend left] (14,5);
                    \draw [dotted] (12,0) -- (12,5);
                    \path (11.7,0)--(11.7,0) node [ below]{$b_{\calD^{\geq 2}}$};
                    \draw  (12.5,2.5)--(15,2.5)--(15,3.5)--(14,3.5)--(14,4)--(12.5,4)--cycle;

                    \node [right] at (13.2,3.3) {$\mathcal{Z}_4$};

                    \draw  (12.5,0)--(15,0)--(15,2.5)--(12.5,2.5)--cycle;

                    \draw [fill=gray!50] (12.5,0)--(15,0)--(15,2.5)--(12.5,2.5)--cycle;
                    \node [right] at (13.2,1.2) {$\mathcal{Z}_5$};
                    \draw   (15,0)--(17,0)--(17,1)--(16.5,1)--(16.5,2)--(15,2)--cycle;

                    \draw [fill=gray!50] (15,0)--(17,0)--(17,1)--(16.5,1)--(16.5,2)--(15,2)--cycle;
                    \draw (15.9,1.2) node {$\mathcal{Z}_6$};  
                \end{tikzpicture}}
        \end{center}
        \caption{$\calH^1$ }\label{H}
    \end{figure}

    \begin{figure}[htb] 
        \begin{center}
            \scalebox{1.4}{
                \begin{tikzpicture}[thick, scale=0.5, every node/.style={scale=0.6}]
                    \draw (7,0) node [below right]{$y$};
                    \draw (0,6) node [above left]{$z$};  

                    \path (0,0)--(0,4)  node [below left]{$\gamma$};

                    \path (0,-0.2)--(5,-0.2) node [below left]{$\beta$};

                    \draw (0,4)--(2,4)--(2,5)--(1.5,5)--(1.5,5.5)--(0.5,5.5)--(0.5,6)--(0,6)--cycle;
                    \draw [fill=gray!50](0,4)--(2,4)--(2,5)--(1.5,5)--(1.5,5.5)--(0.5,5.5)--(0.5,6)--(0,6)--cycle;
                    \node [right] at (0.5,4.7) {$\mathcal{Z}_1$};
                    \draw [dotted] (0,5) -- (1.5,5);
                    \path (0,4.7)--(0,4.7) node [ left]{$c_{\calD^{\geq 2}}$};
                    \draw (0,2.5)--(2.5,2.5)--(2.5,4)--(0,4)--cycle;
                    \node [right] at (0.5,3.3) {$\mathcal{Z}_2$};

                    \draw  (0,0)--(2.5,0)--(2.5,2.5)--(0,2.5)--cycle;
                    \draw [fill=gray!50](0,0)--(2.5,0)--(2.5,2.5)--(0,2.5)--cycle;
                    \node [right] at (0.5,1.2) {$\mathcal{Z}_3$};

                    \draw [pattern=crosshatch] (2.5,2.5)--(2,2.5)--(2,2)--(2.5,2)--cycle; 
                    \node [above right] at (4,5) {$(j_0,k_0)$}; 
                    \draw [->, thin] (2.25,2.25) to [bend left] (4,5);
                    \draw [dotted] (2,0) -- (2,5);
                    \path (1.7,0)--(1.7,0) node [ below]{$b_{\calD^{\geq 2}}$};
                    \draw (2.5,2.5)--(5,2.5)--(5,3.5)--(4,3.5)--(4,4)--(2.5,4)--cycle;
                    \node [right] at (3.2,3.3) {$\mathcal{Z}_4$};

                    \draw (2.5,0)--(5,0)--(5,2.5)--(2.5,2.5)--cycle;

                    \draw [fill=gray!50](2.5,0)--(5,0)--(5,2.5)--(2.5,2.5)--cycle;
                    \node [right] at (3.2,1.2) {$\mathcal{Z}_5$};
                    \draw (5,0)--(7,0)--(7,1)--(6.5,1)--(6.5,2)--(5,2)--cycle;

                    \draw [fill=gray!50](5,0)--(7,0)--(7,1)--(6.5,1)--(6.5,2)--(5,2)--cycle;
                    \draw (5.9,1.2) node {$\mathcal{Z}_6$};

                    \draw (17,0) node [below right]{$y$};
                    \draw (10,6) node [above left]{$z$};
                    \draw [pattern=crosshatch] (12.5,2.5)--(12,2.5)--(12,2)--(12.5,2)--cycle; 
                    \node [above right] at (14,5) {$(j_0,k_0)$}; 
                    \draw [->, thin] (12.25,2.25) to [bend left] (14,5);     

                    \draw (17,0) node [below right]{$y$};
                    \draw (10,6) node [above left]{$z$};  

                    \path (10,0)--(10,4)  node [below left]{$\gamma$};

                    \path (10,-0.2)--(15,-0.2) node [below left]{$\beta$};

                    \draw  (10,4)--(12,4)--(12,5)--(11.5,5)--(11.5,5.5)--(10.5,5.5)--(10.5,6)--(10,6)--cycle;
                    \draw [fill=gray!50](10,4)--(10,6)--(10.5,6)--(10.5,5.5)--(11.5,5.5)--(11.5,5)--(12,5)--(12,4)--cycle;
                    \node [right] at (10.5,4.7) {$\mathcal{Z}_1$};
                    \draw [dotted] (10,3) -- (11.5,3);
                    \path (10,2.7)--(10,2.7) node [ left]{$c_{\calD^{\geq 2}}$};
                    \draw (10,2.5)--(12.5,2.5)--(12.5,4)--(10,4)--cycle;
                    \draw [fill=gray!50](10,4)--(12.5,4)--(12.5,3)--(10,3)--cycle;
                    \node [right] at (10.5,3.3) {$\mathcal{Z}_2$};

                    \draw  (10,0)--(12.5,0)--(12.5,2.5)--(10,2.5)--cycle;
                    \draw [fill=gray!50](10,0)--(12.5,0)--(12.5,2.5)--(10,2.5)--cycle;
                    \node [right] at (10.5,1.2) {$\mathcal{Z}_3$};

                    \draw [pattern=crosshatch] (12.5,2.5)--(12,2.5)--(12,2)--(12.5,2)--cycle; 
                    \node [above right] at (14,5) {$(j_0,k_0)$}; 
                    \draw [->, thin] (12.25,2.25) to [bend left] (14,5);
                    \draw [dotted] (12,0) -- (12,5);
                    \path (11.7,0)--(11.7,0) node [ below]{$b_{\calD^{\geq 2}}$};
                    \draw (12.5,2.5)--(15,2.5)--(15,3.5)--(14,3.5)--(14,4)--(12.5,4)--cycle;
                    \node [right] at (13.2,3.3) {$\mathcal{Z}_4$};

                    \draw (12.5,0)--(15,0)--(15,2.5)--(12.5,2.5)--cycle;

                    \draw [fill=gray!50](12.5,0)--(15,0)--(15,2.5)--(12.5,2.5)--cycle;
                    \node [right] at (13.2,1.2) {$\mathcal{Z}_5$};
                    \draw (15,0)--(17,0)--(17,1)--(16.5,1)--(16.5,2)--(15,2)--cycle;

                    \draw [fill=gray!50](15,0)--(17,0)--(17,1)--(16.5,1)--(16.5,2)--(15,2)--cycle;
                    \draw (15.9,1.2) node {$\mathcal{Z}_6$};  
                \end{tikzpicture}}
        \end{center}
        \caption{$\widetilde{\calD}^1$}\label{D}
    \end{figure}

\end{Example}

The following lemma is used for proving \Cref{thm-reg} when we deal with the points in the second stage of diagram $\calD$. It is also needed by \Cref{reg-lem} in a special case.

\begin{Lemma}
    \label{lex-layer-1}
    Let $\calD$ be a three-dimensional Ferrers diagram that satisfies the projection property. 
    Suppose that $b_\calD>b_{\calD^{\ge 2}}$.
    Let $\calL=\calA_{\bdw}(\calD)$ with respect to the lexicographic order for the point $\bdw=(1,b_{\calD^{\ge 2}}+1,1)\in \calD$.
    Suppose that for each $\bdv\in \calL^1$, the restriction complex $\Delta(\calD,\calA_\bdv(\calD))$ is $\Delta(\calA_\bdv(\calD))$. Furthermore, suppose that $\Delta(\calL)=\Delta(\calD,\calL)$ is pure vertex-decomposable and the lexicographic order gives a shedding order (at least for the points in the first layer). Then, the regularity of $\Delta(\calL)$ is bounded above by $\mu_{\calD}-3$.
\end{Lemma}

\begin{proof}
    It suffices to prove by backward induction with respect to the lexicographic order that
    \begin{equation}
        \reg(\Delta(\calA_{\bdu}(\calD)))\le \mu_{\calD}-3
        \label{lex-1}
    \end{equation}
    for each $\bdu=(1,j_0,k_0)\in \calL^1$.
    \begin{flushleft}
        \textbf{Base case}
        \par\end{flushleft}
    The base case for this induction process is when we remove $\calL^1$ and get $\calL^{\ge 2}=\calD^{\ge 2}$. 
    Since $a_{\calD}>a_{\calD^{\ge 2}}$ and $b_{\calD}>b_{\calD^{\ge 2}}$,
    we have $\mu_{\calD}>\mu_{\calD^{\ge 2}}$. Again, as $a_{\calD^{\ge 2}}<a_{\calD}$, by induction, it is legal to apply \eqref{reg-bound} to get 
    \[
        \reg(\Delta(\calD^{\ge 2}))\le \mu_{\calD^{\ge 2}}-2 <\mu_{\calD}-2,
    \]
    confirming \eqref{lex-1} in this base case.
    \begin{flushleft}
        \textbf{Induction step}
        \par\end{flushleft}
    Consider a general $\bdu\in \calL^1$.
    As in the proof of \Cref{induction-order-strong}, we may assume that $\bdu$ is a normal point with respect to  $\calD$ and try to prove that
    \begin{equation}
        \reg(\link_{\Delta(\calA_{\bdu}(\calD))}(T_{\bdu}))\le \mu_{\calD}-4.
        \label{eqn:link-stange-2}
    \end{equation}
    Similarly, one can introduce $\calH$ with respect to $\calD$ and $\bdu$. To be more precise, here
    \[
        \calH:=\calZ_1^{\ge 2}(\calD,\bdu) \cup \calZ_3^{\ge 2}(\calD,\bdu) \cup \calZ_5^1(\calD,\bdu)\cup \calZ_6^1 (\calD,\bdu);
    \]
    unlike in \eqref{eqn:lex-H}, we have an additional $\calZ_1^{\ge 2}(\calD,\bdu)$, since we don't assume strong projection property here.
    Whence, $\link_{\Delta(\calA_{\bdu}(\calD))}(T_{\bdu})$ is the join of a simplex with $\Delta(\calD,\calH)=\Delta(\calH)$.
    The expected inequality \eqref{eqn:link-stange-2} is then equivalent to 
    \begin{equation}
        \reg(\Delta(\calH))\le \mu_{\calD}-4.
        \label{lex-3}
    \end{equation}
    Notice that for each lattice point $(i,j,k)\in \calZ_1^{\ge 2}(\calD,\bdu) \cup \calZ_3^{\ge 2}(\calD,\bdu)$, we have indeed that $j\le b_{\calD^{\ge 2}}<j_0$. Thus, an ambient concise diagram for $\calH$ can be chosen as
    \[
        \calD':=\Set{(i,j,k)\in \calZ_1(\calD,\bdu) \cup \calZ_3(\calD,\bdu): j<j_0} \cup \calZ_5^1(\calD,\bdu)\cup \calZ_6^1 (\calD,\bdu). 
    \]

    This set is \emph{essentially} a three-dimensional Ferrers diagram that still satisfies the projection property.  Notice that 
    \[
        \calH=
        \begin{dcases*}
            \calA_{(1,j_0+1,1)}(\calD'), & if $\calZ_5^1(\calD,\bdu)\cup \calZ_6^1 (\calD,\bdu)\ne \varnothing$,\\
            \calA_{(2,1,1)}(\calD'),& otherwise.       
        \end{dcases*}
    \]
    Furthermore, as $\bdu$ is not a phantom point and $j_0>b_{\calD^{\ge 2}}$, we have $\bdu\in \calB_z\cup (\calD^1\setminus \calB)$ by \Cref{dis3.7}. Consequently,
    $k_0<\gamma_{\calD}(\bdu)$. Thus, $a_{\calD'}=a_{\calD}$, $b_{\calD'}=b_{\calD}-1$ and $c_{\calD'}=c_{\calD}-(\gamma_{\calD}(\bdu)-k_0)<c_{\calD}$. Now, by induction, we can apply \eqref{lex-1} to get
    \[
        \reg(\Delta(\calH))\le \mu_{\calD'}-3<\mu_{\calD}-3,
    \]
    confirming the expected inequality \eqref{lex-3} in this case. This completes our proof of \Cref{lex-layer-1}.
\end{proof}

In the following lemma,
we estimate the regularity associated with the set $\calH$ given by \eqref{eqn:complicated-H}. The proof is similar to that for \Cref{thm-reg}; we still use backward induction with respect to the order introduced in \Cref{set:IO}. Along the proof, we have to reduce the ambient diagram $\widetilde{D}$ to a smaller diagram called $\overline{\calD}$. 
We invite the readers to use the figures regarding $\overline{D}$ in \Cref{ex:overlineD} for reference when reading the proof of \Cref{reg-lem}.   

\begin{Lemma}
    \label{reg-lem}
    Under the assumptions in \Cref{thm-reg}, suppose that $\bdu=(1,j_0,k_0) \in \calD^1$ is a normal point with $k_0\le c_{\calD^\ge 2}$. Let $\calH$ and $\widetilde{\calD}$ be the sets given by \eqref{eqn:complicated-H} and \eqref{eqn:tilde-D} respectively.
    If $\calH^1\ne \varnothing$, then 
    \begin{equation}
        \reg(\Delta(\calH)) \le \mu_{\widetilde{\calD}}-3.
        \label{eqn:reg-ineq-H}
    \end{equation}
\end{Lemma}

\begin{proof}
    For notational simplicity, we may assume that $k_0=\min(\gamma_{\calD}(\bdu),c_{\calD^{\ge 2}})$.  Hence the ambient diagram $\widetilde{\calD}$ is indeed a three-dimensional Ferrers diagram and $c_{\widetilde{\calD}^{\ge 2}}=c_{\calD^{\ge 2}}\ge k_0$. 
    Let $\prec_{\widetilde{\calD}}$ be the induction order with respect to $\widetilde{\calD}$; see also \Cref{set:IO}.
    The subset of $\calH^1$ in the \emph{first stage}, namely $\calC_1(\widetilde{\calD})\cap \calH$, is precisely $\calZ_5^1({\calD},\bdu)\cup \calZ_6^1({\calD},\bdu)$. 
    Certainly, the subset of $\calH^1$ in the \emph{second stage}, namely $\calC_2(\widetilde{\calD})\cap \calH$, is given by $\Set{(1,j,k)\in \widetilde{\calD} : j\le j_0 \text{ and } k>c_{\widetilde{\calD}^{\ge 2}}}$.
    Notice that $\calH=\calA_\bdu^+(\widetilde{\calD})=\calA_{\bdv_0}(\widetilde{\calD})$ for some $\bdv_0$. The following observations are clear.
    \begin{enumerate}[a]
        \item If $\calH^1\cap \calC_1(\widetilde{\calD}) \ne \varnothing$, then  $\bdv_0=(1,j_0+1,1)$ is the initial point of $\calH^1\cap \calC_1(\widetilde{\calD})$. 
        \item If $\calH^1\cap \calC_1(\widetilde{\calD}) = \varnothing$ while $\calH^1\cap \calC_2(\widetilde{\calD}) \ne \varnothing$, then
            $\bdv_0=(1,1,c_{\widetilde{\calD}^{\ge 2}}+1)$ is the initial point of $\calH^1\cap \calC_2(\widetilde{\calD})$.
        \item If $\calH^1=\varnothing$, then  $\bdv_0=(2,1,1)$ is the initial point of $\widetilde{\calD}^{\ge 2}$.
    \end{enumerate}
    Hence, to achieve \eqref{eqn:reg-ineq-H}, it suffices to prove 
    \begin{equation}
        \reg(\Delta(\calA_{\bdv}(\widetilde{\calD}))) \le \mu_{\widetilde{\calD}}-3
        \label{eqn-7}
    \end{equation}
    for each $\bdv\in \calH^1$ in the \emph{first stage} with respect to $\widetilde{D}$. We will prove by backward induction with respect to the induction order $\prec_{\widetilde{\calD}}$.
    \begin{flushleft}
        \textbf{Base case}
        \par\end{flushleft}
    We have two subcases to check.
    \begin{enumerate}[a]
        \item Suppose that the points of the second stage exist, i.e., $\calC_2(\widetilde{\calD})\ne \varnothing$. Then, we need to prove \eqref{eqn-7} for $\bdv=(1,1,c_{\widetilde{\calD}^{\ge 2}}+1)$, the initial element in the second stage.  For that purpose, we flip $\widetilde{\calD}$ to get $\calS(\widetilde{\calD})$ and write it as $\widetilde{\calD}'$ for simplicity.  Then, $\calS(\calA_\bdv(\widetilde{\calD}))$ is precisely $\calA_{(1,b_{(\widetilde{\calD}')^{\ge 2}}+1,1)}(\widetilde{\calD}')$ (the $\calA$ in $\calA_{(1,b_{(\widetilde{\calD}')^{\ge 2}}+1,1)}(\widetilde{\calD}')$ is with respect to the lexicographic order on $\widetilde{\calD}'$).  Since $b_{(\widetilde{\calD}')^{\ge 2}}<b_{\widetilde{\calD}'}$ while $\mu_{\widetilde{\calD}}=\mu_{\widetilde{\calD}'}$,  we can apply \Cref{lex-layer-1} to complete the proof (the conditions of \Cref{lex-layer-1} are satisfied, as mentioned earlier in the proof of \Cref{thm-reg}).   

        \item Suppose instead that $\calC_2(\widetilde{\calD}) = \varnothing$. Then, we need to prove \eqref{eqn-7} for $\bdv=(2,1,1)$, the initial element of the second layer. Since $\calH^1\ne \varnothing$, either $\calZ_5^1(\calD,\bdu)\cup \calZ_6^1(\calD,\bdu)\ne \varnothing$ (hence $b_{\widetilde{\calD}}>b_{\widetilde{\calD}^{\ge 2}}$) or 
            \begin{equation}
                c_{\widetilde{\calD}}>\min(\gamma,c_{\calD^{\ge 2}})=k_0,
                \label{eqn-8}
            \end{equation}
            for $\gamma=\gamma_{\calD}(\bdu)$.
            Notice that since $b_{\widetilde{\calD}}\ge j_0\ge b_{\widetilde{\calD}^{\ge 2}}$, if $b_{\widetilde{\calD}}=b_{\widetilde{\calD}^{\ge 2}}$, then they coincide with $j_0$. Whence, by the projection property of $\widetilde{\calD}$, one must have $\gamma\ge c_{\widetilde{\calD}^{\ge 2}}=c_{\calD^{\ge 2}}$ (the last equality holds since we have assumed that $k_0=\min(\gamma,c_{\calD^{\ge 2}})$ for simplicity). Thus, the inequality in \eqref{eqn-8} will then imply that $c_{\widetilde{\calD}}>c_{\widetilde{\calD}^{\ge 2}}$. 
            In short, we have either $b_{\widetilde{\calD}}>b_{\widetilde{\calD}^{\ge 2}}$ and $c_{\widetilde{\calD}}\ge c_{\widetilde{\calD}^{\ge 2}}$, or $b_{\widetilde{\calD}}\ge b_{\widetilde{\calD}^{\ge 2}}$ and $c_{\widetilde{\calD}}>c_{\widetilde{\calD}^{\ge 2}}$.
            Since $a_{\widetilde{\calD}^{\ge 2}}=a_{\calD}-1$, in either case, one has $\mu_{\widetilde{\calD}^{\ge 2}}<\mu_{\widetilde{\calD}}$, as long as $\calH^1\ne \varnothing$.

            Again, since $a_{\widetilde{\calD}^{\ge 2}}=a_{\calD}-1$, we can assume
            \Cref{thm-reg} for $\widetilde{\calD}^{\ge 2}$ to get
            \[
                \reg(\Delta(\widetilde{\calD}^{\ge 2}))\le \mu_{\widetilde{\calD}^{\ge 2}}-2\le \mu_{\widetilde{\calD}}-3.
            \]  
            Hence the expected inequality \eqref{eqn-7} holds in this subcase. 
    \end{enumerate}
    \begin{flushleft}
        \textbf{Induction step}
        \par\end{flushleft}
    Consider a general $\bdv=(1,j_1,k_1) \in \calH\cap \calC_1(\widetilde{\calD})$.  
    As in the proof of \Cref{thm-reg}, by induction, the proof is reduced to the case when $\bdv$ is a normal point and we only need to show
    \begin{equation*}
        \reg(\link_{\Delta(\calA_{\bdv}(\widetilde{\calD}))}(T_{\bdv}))\le \mu_{\widetilde{\calD}}-4.
    \end{equation*}
    Just as $\calH$ for $\calD$ and $\bdu$, we introduce similarly $\widetilde{\calH}$ for $\widetilde{\calD}$ and $\bdv$. Then the previously expected inequality is simply 
    \begin{equation}
        \reg(\Delta(\widetilde{\calH}))\le \mu_{\widetilde{\calD}}-4.
        \label{eqn:tilde-H}
    \end{equation}
    More explicitly,
    \begin{align*}
        \widetilde{\calH}
        &\coloneqq \calZ_1(\widetilde{D},\bdv)^{\ge 2} \cup \calZ_3(\widetilde{D},\bdv)^{\ge 2} \cup \calZ_5(\widetilde{D},\bdv)^1\cup \calZ_6(\widetilde{D},\bdv)^1 
        \cup \Set{(1,j,k)\in \widetilde{\calD}: j\le j_1 \text{ and }
            k>c_{\widetilde{\calD}^{\ge 2}}}
    \end{align*}
    as in \eqref{eqn:complicated-H}.
    Write $\gamma(\bdv)$ for $\gamma_{\widetilde{\calD}}(\bdv)$.  As $k_1\le \gamma(\bdv)\le k_0\le c_{\widetilde{\calD}^{\ge 2}}=c_{\calD^{\ge 2}}$, we have 
    \begin{align*}
        \Set{(1,j,k)\in \widetilde{\calD}: j\le j_1 \text{ and } k>c_{\widetilde{\calD}^{\ge 2}}} 
        &\subseteq  \Set{(1,j,k)\in \widetilde{\calD}: j\le j_1 \text{ and } k>\gamma(\bdv)} \\
        &=  \Set{(1,j,k)\in \widetilde{\calD}: j< j_1 \text{ and } k>\gamma(\bdv)}.
    \end{align*}
    Furthermore, notice that from \eqref{eqn:tilde-D}, one can say safely that for any $(i,j,k)\in \calZ_1(\widetilde{D},\bdv)^{\ge 2} \cup \calZ_3(\widetilde{D},\bdv)^{\ge 2}$, we have $j\le j_0<j_1$.  Thus, the ambient diagram for $\widetilde{\calH}$ can be chosen as the more concise subset
    \begin{align}
        \overline{\calD} &\coloneqq \calZ_1(\widetilde{\calD},(1,j_1-1,k_1)) \cup
        \calZ_3(\widetilde{\calD},(1,j_1-1,k_1)) \cup \calZ_5(\widetilde{\calD},\bdv)^1 \cup \calZ_6(\widetilde{\calD},\bdv)^1
        \notag \\
        & \qquad \cup \Set{(i,j,k)\in \widetilde{\calD}: j< j_1 \text{ and } k>\gamma(\bdv)}. \label{eqn:D-bar}
    \end{align}
    Again, this diagram is essentially a three-dimensional Ferrers diagram that still satisfies the projection property. 

    If $k_1=\gamma(\bdv)$, as $j_1>j_0\ge b_{\widetilde{\calD}^{\ge 2}}$, $\bdv$ is a phantom point of $\widetilde{\calD}$ by \Cref{dis3.7}. But this is against our assumption on $\bdv$.
    If $k_1<\gamma(\bdv)$, then $c_{\overline{\calD}}=c_{\widetilde{\calD}}-(\gamma(\bdv)-k_1)<c_{\widetilde{\calD}}$. Meanwhile, $b_{\overline{\calD}}=b_{\widetilde{\calD}}-1$. Thus, $\mu_{\overline{\calD}}\le \mu_{\widetilde{\calD}}-1$.
    Since $\overline{\calD}$ is a smaller diagram than $\widetilde{\calD}$, by induction, we can apply the inequality \eqref{eqn:reg-ineq-H} to $\widetilde{\calH}$ with respect to $\overline{\calD}$ to get  
    \begin{equation*}
        \reg(\Delta(\widetilde{\calH}))\le \mu_{\overline{\calD}}-3 \le \mu_{\widetilde{\calD}}-4,
    \end{equation*}
    confirming the expected inequality \eqref{eqn:tilde-H}.
    And this completes our proof of \Cref{reg-lem}. \qedhere   
\end{proof}

\begin{Example}\label{ex:overlineD}
    In the proof of \Cref{reg-lem}, we introduced a concise subdiagram $\overline{\calD}$ in \eqref{eqn:D-bar}.
    As a continuation to \Cref{exam:6.2}, in \Cref{overlineD} we present typical $\overline{\calD}^1$ by the shaded regions with two different cases: 
    $\gamma\leq c_{\calD^{\ge 2}}$ (left-hand side) and $\gamma>c_{\calD^{\ge 2}}$ (right-hand side).
    These diagrams are deduced from the $\widetilde{D}$ in \Cref{exam:6.2}. Unlike in the proof \Cref{reg-lem}, we don't assume that $k_0=\min(\gamma_{\calD}(\bdu),c_{\calD^{\ge 2}})$ here.
    \begin{figure}[htb] 
        \begin{center}
            \scalebox{1.4}{
                \begin{tikzpicture}[thick, scale=0.5, every node/.style={scale=0.6}]
                    \draw (7,0) node [below right]{$y$};
                    \draw (0,6) node [above left]{$z$};  

                    \path (0,0)--(0,4)  node [below left]{$\gamma$};
                    \path (-1,0)--(-1,4)  node [below left]{$\widetilde{\calD}^1$};
                    \path (0,-0.2)--(5,-0.2) node [below left]{$\beta$};

                    \draw (0,4)--(2,4)--(2,5)--(1.5,5)--(1.5,5.5)--(0.5,5.5)--(0.5,6)--(0,6)--cycle;
                    \draw [fill=gray!50](0,4)--(2,4)--(2,5)--(1.5,5)--(1.5,5.5)--(0.5,5.5)--(0.5,6)--(0,6)--cycle; 
                    \draw [dotted] (0,5) -- (1.5,5);
                    \path (0,4.7)--(0,4.7) node [ left]{$c_{\calD^{\geq 2}}$};

                    \draw [fill=gray!50](0,0)--(3,0)--(3,1.5)--(0,1.5)--cycle; 
                    \draw [pattern=crosshatch] (2.5,2.5)--(2,2.5)--(2,2)--(2.5,2)--cycle; 
                    \node [above right] at (4,5) {$(j_0,k_0)$}; 
                    \draw [->, thin] (2.25,2.25) to [bend left] (4,5);
                    \draw [dotted] (2,0) -- (2,5);
                    \path (1.7,0)--(1.7,0) node [ below]{$b_{\calD^{\geq 2}}$};
                    \draw (5,2)--(5,3.5)--(4,3.5)--(4,4)--(2,4); 

                    \draw (0,4)--(0,0)--(5,0);
                    \draw [dotted] (0,2.5)--(5,2.5);
                    \draw [dotted] (2.5,0)--(2.5,4);
                    \draw [dotted] (3.5,1.5)--(3.5,4);

                    \draw [fill=gray!50](3.5,0)--(7,0)--(7,1)--(6.5,1)--(6.5,1.5)--(3.5,1.5)--cycle; 
                    \draw (6.5,0)--(6.5,1);
                    \draw (5,0)--(7,0)--(7,1)--(6.5,1)--(6.5,2)--(5,2);
                    \draw [dotted] (5,2)--(5,0);

                    \draw [pattern=crosshatch] (3,1)--(3,1.5)--(3.5,1.5)--(3.5,1)--cycle; 
                    \node [above right] at (4,-1.75) {$(j_1,k_1)$}; 
                    \draw [->, thin] (3.25,1.25) to [bend right] (4,-1.5);

                    \draw (17,0) node [below right]{$y$};
                    \draw (10,6) node [above left]{$z$};

                    \draw [pattern=crosshatch] (12.5,2.5)--(12,2.5)--(12,2)--(12.5,2)--cycle; 
                    \node [above right] at (14,5) {$(j_0,k_0)$}; 
                    \draw [->, thin] (12.25,2.25) to [bend left] (14,5);     

                    \draw (17,0) node [below right]{$y$};
                    \draw (10,6) node [above left]{$z$};  

                    \path (10,0)--(10,4)  node [below left]{$\gamma$};
                    \path (9,0)--(9,4)  node [below left]{$\widetilde{\calD}^1$};
                    \path (10,-0.2)--(15,-0.2) node [below left]{$\beta$};

                    \draw [fill=gray!50](10,3)--(10,6)--(10.5,6)--(10.5,5.5)--(11.5,5.5)--(11.5,5)--(12,5)--(12,4)--(12.5,4)--(12.5,3)--cycle;
                    \draw [dotted] (10,4)--(12,4);

                    \draw [dotted] (10,3) -- (11.5,3);
                    \path (10,2.7)--(10,2.7) node [ left]{$c_{\calD^{\geq 2}}$};

                    \draw [fill=gray!50](10,0)--(13,0)--(13,1.5)--(10,1.5)--cycle;
                    \draw [dotted] (13.5,4)--(13.5,1.5);
                    \draw [dotted] (12.5,0)--(12.5,3);

                    \draw (10,3)--(10,0)--(15,0);
                    \draw [dotted] (10,2.5)--(15,2.5);

                    \draw [pattern=crosshatch] (12.5,2.5)--(12,2.5)--(12,2)--(12.5,2)--cycle; 
                    \node [above right] at (14,5) {$(j_0,k_0)$}; 
                    \draw [->, thin] (12.25,2.25) to [bend left] (14,5);

                    \draw [dotted] (12,0) -- (12,5);
                    \path (11.7,0)--(11.7,0) node [ below]{$b_{\calD^{\geq 2}}$};
                    \draw (15,2)--(15,3.5)--(14,3.5)--(14,4)--(12.5,4);
                    \draw (15,0)--(17,0)--(17,1)--(16.5,1)--(16.5,2)--(15,2); 
                    \draw [fill=gray!50](13.5,0)--(13.5,1.5)--(16.5,1.5)--(16.5,1)--(17,1)--(17,0)--cycle;
                    \draw [dotted] (15,0)--(15,2);
                    \draw (16.5,0)--(16.5,1);

                    \draw [pattern=crosshatch] (13,1)--(13,1.5)--(13.5,1.5)--(13.5,1)--cycle; 
                    \node [above right] at (14,-1.75) {$(j_1,k_1)$}; 
                    \draw [->, thin] (13.25,1.25) to [bend right] (14,-1.5); 
                \end{tikzpicture}}
        \end{center}
        \caption{$\overline{\calD}^1$}\label{overlineD}
    \end{figure}
\end{Example}

Finally, we wrap up this paper with some quick observations.  Recall that a simplicial complex $\Delta$ is called \emph{acyclic} (over $\KK$) if all the reduced simplicial homology groups $\widetilde{H}_i(\Delta)$ are trivial. Cones are known to be acyclic.

\begin{Corollary}
    Let $\calD$ be a three-dimensional Ferrers diagram that satisfies the projection property. 
    Then the associated Stanley--Reisner complex $\Delta(\calD)$ is acyclic.
\end{Corollary}
\begin{proof}
    Notice that if the Stanley--Reisner ring $\KK[\Delta]$ of a simplicial complex $\Delta$ of dimension $d$ is Cohen--Macaulay, then by the Reisner's criterion \cite[Theorem 8.1.6]{MR2724673}, $\widetilde{H}_i(\Delta)=0$ for all $i<d$. Meanwhile, it follows from Hochster’s formula on the local cohomology modules \cite[Theorem A.7.3]{MR2724673} that $\reg(\KK[\Delta])=\max\Set{j:\widetilde{H}_{j-1}(\link_\Delta(F;\KK))\ne 0 \text{ for some $F\in \Delta$}}$; see also the proof of \cite[Proposition 8.1.10]{MR2724673}. Whence, $\reg(\KK[\Delta])\le  d=(\dim(\KK[\Delta])-1)$ if and only if $\widetilde{H}_d(\Delta)=0$.

    We have shown in \cite[Theorem 4.1]{arXiv:1709.03251} that $\Delta(\calD)$ is Cohen--Macaulay of dimension $a_\calD+b_\calD+c_\calD-3$. Hence, the claimed result follows from \Cref{thm-reg}.
\end{proof}

\begin{acknowledgment*}
    The authors sincerely thank the patient reviewer for the very helpful and constructive suggestions that greatly improve the presentation of this manuscript. The second author was partially supported by the ``Anhui Initiative in Quantum Information Technologies'' (No.~AHY150200) and ``Fundamental Research Funds for the Central Universities''.
\end{acknowledgment*}


\end{document}